\documentclass[leqno, 12pt]{amsart}

\usepackage[utf8]{inputenc}
\usepackage[T1]{fontenc}
\usepackage{CJKutf8}
\usepackage[english]{babel}
\usepackage{graphicx}
\usepackage{amsthm}
\usepackage{amssymb}
\usepackage{amsmath}
\usepackage{mathtools}
\usepackage[colorlinks=true]{hyperref}
\usepackage{csquotes}
\usepackage[margin=1in]{geometry}
\usepackage{enumitem}
\usepackage{float}
\usepackage[svgnames]{xcolor}
\usepackage{subcaption}
\usepackage{soul} 

\definecolor{color1}{RGB}{27,158,119}
\definecolor{color2}{RGB}{217,95,2}
\definecolor{color3}{RGB}{117,112,179}
\definecolor{color4}{RGB}{231,41,138}

\hypersetup{
  pdftitle={},
  pdfauthor={Stefan Le Coz <slecoz@math.univ-toulouse.fr>},
  pdfsubject={},
  pdfkeywords={}, 
  colorlinks=true,
  linkcolor=DarkBlue,
  citecolor=DarkRed,
  filecolor=DarkMagenta,
  urlcolor=DarkGreen,
}

\newtheorem{theorem}{Theorem}[section]

\newtheorem{lemma}[theorem]{Lemma}
\newtheorem{proposition}[theorem]{Proposition}
\theoremstyle{definition}

\theoremstyle{remark}
\newtheorem{remark}[theorem]{Remark}
\theoremstyle{remark}

\numberwithin{equation}{section}

\DeclareMathOperator\R{\mathbb R}
\DeclareMathOperator\C{\mathbb C}

\title[Ground States of NLS on the Tadpole Graph with a Repulsive Delta]{Ground States of the Nonlinear Schrödinger Equation on the Tadpole Graph with a Repulsive Delta Vertex Condition}

\author[R.~Duboscq]{Romain Duboscq}
\author[\'E.~Durand-Simonnet]{Élio Durand-Simonnet}
\author[S.~Le Coz]{Stefan Le Coz}

\thanks{This work was supported by the ANR LabEx CIMI (grant ANR-11-LABX-0040) within the French State Programme "Investissements d'Avenir" and the ANR project NQG ANR-23-CE40-0005.}

\address[Romain Duboscq]{Institut de Math\'ematiques de Toulouse; UMR5219,
  \newline\indent
  Universit\'e de Toulouse; CNRS,
  \newline\indent
  UPS IMT, F-31062 Toulouse Cedex 9,
  \newline\indent
  France}
\email[Romain Duboscq]{romain.duboscq@math.univ-toulouse.fr}

\address[Élio Durand-Simonnet and Stefan Le Coz]{Institut de Math\'ematiques de Toulouse; UMR5219,
  \newline\indent
  Universit\'e de Toulouse; CNRS,
  \newline\indent
  UPS IMT, F-31062 Toulouse Cedex 9,
  \newline\indent
  France}
  \email[Élio Durand-Simonnet]{elio.durand\_simonnet@math.univ-toulouse.fr}
\email[Stefan Le Coz]{slecoz@math.univ-toulouse.fr}

\subjclass[2020]{}

\date{\today}

\begin{document}

\begin{abstract}
We consider the stationary nonlinear Schr\"odinger equation set on a tadpole graph with a repulsive delta vertex condition between the loop and the tail of the tadpole. We establish the existence of an action ground state when the size of the loop is either very small or very large. Our analysis relies on variational arguments, such as profile decomposition. When it exists, we study the shape of the ground state using ordinary differential equations arguments, such as the study of period functions. The theoretical results are completed with a numerical study.
\end{abstract}

\maketitle

\section{Introduction}

In the last fifteen years, the study of nonlinear Schr\"odinger equations set on metric graphs 
has gained an incredible momentum. One of the main issues under investigation is the existence and properties of stationary states, which can be obtained in various ways, e.g. as solutions of coupled ordinary differential equations, as minimizers of the Schr\"odinger energy at fixed mass, as minimizer of the action, etc. In general, the setting is either a general graph with the classical Kirchhoff conditions at the vertices, or a star graph with more generic conditions on the vertex such as $\delta$ or $\delta'$-vertex conditions. Outside the case of the star graph, the influence of the vertex condition on the analysis of stationary states has rarely been considered, and the goal of the present study is to analyze this influence in other situations, in particular in the case of a graph with a compact core to which  infinite half-lines are attached. The simplest possible example of such a graph is the one of the tadpole graph, and it is the one which will be studied in the present paper. 

The \emph{tadpole} (or \emph{lasso}) graph is a graph formed by one vertex and two  edges, one loop connecting the vertex to itself, and a infinite edge starting from the vertex (see Figure \ref{figTadpole}). The tadpole graph is a model case for the study of the wave propagation on metric graphs. One of the first mathematical studies of the tadpole graph in the context of quantum graph was performed by \cite{Ex97}, motivated by physical models of metal loops connected to electron reservoirs by a single wire, see \cite{Bu85,JaSi94}. 

As for quantum graphs in general, the first challenge posed by the tadpole graph is the understanding of the linear operator set on the graph and its associated dynamics. 
In \cite{Ex97}, Exner analyses the spectral properties of the Laplacian on the tadpole graph with generic vertex conditions and a magnetic field on the loop. 
Dispersive effects on the tadpole graph and the shrinking circle limit were investigated by Ali Mehmeti, Ammari and Nicaise in \cite{AlAmNi17}. 
The linear wave propagation on the damped tadpole graph was analyzed by Ammari, Assel and Dimassi \cite{AmAs24,AmAsDi24} while the linear Schr\"odinger and Airy propagation on looping-edge graphs (generalized tadpoles) were considered by Angulo Pava and Mu\~noz \cite{AnMu24}. 

From the nonlinear point of view, the tadpole graph serves as a model case for a graph with a compact core and semi-infinite edges. As such, it was used by Dovetta and Tarentelli \cite{DoTa20} for the study of the existence of ground states in the mass-critical case when the nonlinearity is concentrated on the compact core (i.e the tadpole loop). Noja, Pelinovski and Shaikhova investigated the bifurcation picture for standing waves (see \cite{NoPeSh15}) and their variational characterization in the $L^2$-critical case (see \cite{NoPe20}). Existence of an energy ground state on a tadpole graph with a Kirchhoff condition has been proved in \cite{AsSeTi15}. More recently, the stability of two-lobe states on the tadpole graph for nonlinear Schr\"odinger equations was considered in \cite{Pa24}. 

Few works are devoted to the nonlinear dynamics. The asymptotic behavior (modified scattering) of symmetric solutions of the cubic nonlinear Schr\"odinger equation on the tapdole graph was established in \cite{Se25}. Working toward a better understanding of  the nonlinear Schr\"odinger dynamics is one of the main existing challenges for the analysis on metric graphs.  

While we restrict the study in the present paper to the elementary tadpole described below, the tadpole graph can be generalized to the so-called flower graph, which is made of a line and several loops attached at the same vertex (see \cite{KaMaPeXi21}). There is now abundant literature on stationary states of nonlinear quantum graphs; we refer to \cite{KaNoPe22} for an overview of the existing results and references. 

In this work, we will be interested in the existence of action ground states on the tadpole graph equipped with a delta condition at the vertex. Our objective is to understand, in this model case, how the delta condition at the vertex influences the existence and shape of standing wave solutions.

Before stating our main result, we introduce some notation (some of which will be made more precise in Section \ref{secPreliminaries}). 

We consider a tadpole graph $\mathcal G_L$ with a compact edge of length $2L$. We equip it with a second order derivative operator $\operatorname{H}_\gamma$ incorporating $\delta$-vertex conditions associated to a parameter $\gamma \in\mathbb R$, i.e. formally $\operatorname{H}_\gamma = -\partial_{xx} + \gamma \delta(\cdot - \operatorname{v})$. We denote by $Q_{\operatorname{H}_\gamma}$ its associated quadratic form. See Section \ref{secPreliminaries} for more details about the Hamiltonian operator.

Let $\omega > 0$. We consider the cubic stationary nonlinear Schrödinger equation with a repulsive non-linearity on $\mathcal{G}_L$, given by
\begin{equation} \label{eqNLS}
    \operatorname{H}_\gamma u + \omega u - |u|^2u = 0.
\end{equation}

We define the \emph{action} $S_{\omega, \gamma, L}$ and the \emph{Nehari functional} $I_{\omega, \gamma, L}$ for $v \in D(Q_{\operatorname{H}_\gamma})$ by
\begin{align}
    S_{\omega, \gamma, L} (v) & = \frac{1}{2} Q_{\operatorname{H}_\gamma}(v) + \frac{\omega}{2} \|v\|_{L^2(\mathcal{G}_L)}^2 - \frac{1}{4} \|v\|_{L^4(\mathcal{G}_L)}^4, \label{eqAction} \\
    I_{\omega, \gamma, L} (v) & = Q_{\operatorname{H}_\gamma}(v) + \omega \|v\|_{L^2(\mathcal{G}_L)}^2 - \|v\|_{L^4(\mathcal{G}_L)}^4, \label{eqNehariFunctional}
\end{align}
and the \emph{Nehari manifold} $\mathcal N_{\omega, \gamma, L}$ by
\begin{equation}
    \mathcal N_{\omega, \gamma, L} = \left\{ v \in D \left( Q_{\operatorname{H}_\gamma} \right) : v \not\equiv 0 \text{ and } I_{\omega, \gamma, L} (v) = 0 \right\}. \label{eqNehariManifold}
\end{equation}
An \emph{action ground state} $u$ is a solution of the variational problem
\begin{equation} \label{eqMinAction}
    s_{\omega, \gamma, L} = \inf \left\{ S_{\omega, \gamma, L} (v) : v \in \mathcal N_{\omega, \gamma, L} \right\}.
\end{equation}
In particular, an action ground state is a solution to \eqref{eqNLS} (see Section \ref{secPreliminaries}).

Our main results can be summarized as follows.

\begin{theorem} \label{thmMain}
    The following statements hold.
    \begin{enumerate}[label=(\roman*)]
        \item Let $\omega > 0$ and $0 < \gamma < ( 1 - \sqrt{3}/2 )^{1/3} \sqrt{\omega}$. Then, there exists $L^* > 0$ such that, for all $0 < L < L^*$, there exists an action ground state.
        \item Let $\omega > 0$ and $0 < \gamma < \sqrt{\omega}$. Then, there exists $L^{**} > 0$ such that, for all $L \geq L^{**}$, there exists an action ground state.
    \end{enumerate}
    Furthermore, the action ground state has a shape that must be chosen from those described in Propositions \ref{propShapeEven} and \ref{propShapeOdd}. %
\end{theorem}

%
%
%

We have chosen in the present study to restrict ourselves to the case of a repulsive delta condition on the vertex. Indeed, it is in this case that we expect to observe the richer phenomenology, the attractive delta case being usually simpler to handle and closer to the Kirchhoff case (see e.g. \cite{LeFuFiKsSi08} in the case of the line with a delta at the origin).

We also have chosen to focus on the analysis of action ground states, i.e. minimizers of the action on the Nehari manifold. The techniques that we have developed could also be useful for the analysis of energy ground states, i.e. minimizers of the Schr\"odinger energy at fixed mass. We refer to the recent work \cite{DeDoGaSe23} for the relations between action and energy ground states in the context of metric graphs.

With our current theoretical technology, we are limited in the analysis to the case of either small or large length  $L$. To further understand the phenomenology in the whole range $L\in(0,\infty)$, we can rely on numerical simulations, which are presented in Section \ref{sec:numerics}. In particular, we observed in numerical simulations the existence of action ground states for any $L\in(0,\infty)$, with data at the vertex positioned on an explicitly given curve ($\Gamma_\omega$ in Figure \ref{figGS3}). 

The rest of this article is divided as follows. In Section \ref{secPreliminaries}, we present the functional setting on the tadpole graph, we introduce the Hamiltonian operator and its associated quadratic form, we make the link between the minimization problem and the equation on the graph and we present the classical ordinary differential setting on which we will rely, including in particular the period functions. Section \ref{sec:ground_states} is devoted to the proof of the existence of ground states in the limit of small and large loop length. We first establish an existence criterion: a ground state will exist on the tadpole if the Nehari action infimum on the tadpole is strictly smaller than the one on the line. This is based on a profile decomposition argument. The existence result is then obtained by constructing a suitable competitor whose action is smaller than the Nehari minimum on the line. In the small limit case, the competitor is close to the half-soliton on the line, while in the large limit case it is close to the full soliton on the loop. In Section \ref{sec:non_negative}, we analyse the various possible shapes for non negative stationary states. To this aim, we consider the curves of possible data constructed using the boundary condition at the vertex. Depending on the position of this curve with respect to the soliton curve on the phase plane, we may obtain different shapes for the stationary states. In Section \ref{sec:numerics} we present numerical experiments which confirm the theoretical and present the type of ground states which are expected depending on the value of the length. The appendix \ref{secSpectrum} presents a detailed spectral analysis of the linear operator and closes the paper. 

\section{Preliminaries} \label{secPreliminaries}

\subsection{Tadpole graph}

Let $L > 0$. We consider the metric tadpole graph $\mathcal{G}_L$ composed of a finite edge $c$, where $c$ stands for \emph{circle}, identified to $[-L, L]$, attached to an infinite edge $h$, where $h$ stands for \emph{half-line}, identified to $[0, \infty)$ at the vertex $\operatorname{v}$ corresponding to $\{-L\}$ and $\{L\}$ on $c$ and to $\{0\}$ on $h$. A schematic representation of $\mathcal{G}_L$ is given in Figure \ref{figTadpole}. A point $\xi$ of $\mathcal{G}_L$ is represented either by $\xi_c \in (-L, L)$ if $\xi \in c\setminus\{\operatorname{v}\}$ or by $\xi_h \in (0, \infty)$ if $\xi \in h\setminus\{\operatorname{v}\}$.
\begin{figure}[ht]
    \centering
    \includegraphics{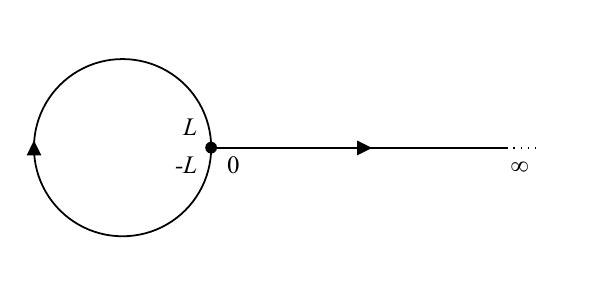}
    \caption{The tadpole graph $\mathcal{G}_L$.}
    \label{figTadpole}
\end{figure}

A function $v : \mathcal{G}_L \to \C$ is a couple of one dimensional maps $(v_c, v_h) : [-L,L] \times [0,\infty) \to \C^2$. Let be $\xi$ a point of $\mathcal{G}_L$, which is not the vertex $\operatorname{v}$, and $v$ a function on $\mathcal{G}_L$. If the context requires it, we allow ourselves to simplify the notation and use $v(\xi)$ to denote
\begin{equation*}
    v (\xi) =
    \begin{cases}
        v_c (\xi_c) & \text{ if } \xi \in c, \\
        v_h (\xi_h) & \text{ if } \xi \in h.
    \end{cases}
\end{equation*}
If $v$ satisfies continuity at the vertex, i.e.
\[
    v_c (-L) = v_c(L) = v_h(0),
\]
we denote by $v(\operatorname{v})$ the common value of $v$ at the vertex $\operatorname{v}$. We define the Lebesgue integral of a function $v$ on the graph $\mathcal{G}_L$ by
\[
    \int_{\mathcal{G}_L} v(\xi) d\xi = \int_{-L}^L v_c(x) dx + \int_0^\infty v_h(x) dx.
\]
For $p, k \in \mathbb N^*$, we define the Lebesgue space $L^p(\mathcal{G}_L)$ and the Sobolev space $H^k(\mathcal{G}_L)$ by
\[
    L^p (\mathcal{G}_L) = L^p(-L, L) \oplus L^p(0, \infty), \quad H^k (\mathcal{G}_L) = H^k(-L, L) \oplus H^k(0, \infty)
\]
endowed with the norms
\[
    \left\| v \right\|_{L^p (\mathcal{G}_L)}^p = \left\| v_c \right\|_{L^p(-L, L)}^p + \left\| v_h \right\|_{L^p(0, \infty)}^p, \quad \left\| v \right\|_{H^k(\mathcal{G}_L)}^2 = \left\| v_c \right\|_{H^k(-L, L)}^2 + \left\| v_h \right\|_{L^k(0, \infty)}^2.
\]
Finally, the scalar product on $L^2(\mathcal{G}_L)$ is given by:
\[
    (u, v)_{L^2(\mathcal{G}_L)} = \operatorname{Re} \left( \int_{\mathcal{G}_L} u(x) \overline{v(x)} dx \right) = \operatorname{Re} \left( \int_{-L}^L u_c(x) \overline{v_c(x)} dx + \int_0^\infty u_h(x) \overline{v_h(x)} dx \right).
\]
For $q \in (2,\infty]$ and $v \in H^1(\mathcal{G}_L)$, the Gagliardo–Nirenberg inequality
\begin{equation} \label{eqGN}
    \| v \|_{L^q(\mathcal{G}_L)} \leq C \| v' \|_{L^2(\mathcal{G}_L)}^{\frac{1}{2}-\frac{1}{q}} \| v \|_{L^2(\mathcal{G}_L)}^{\frac{1}{2}+\frac{1}{q}}
\end{equation}
holds for $C = C(q, L) > 0$.
The proof of \eqref{eqGN} follows from the analogous estimates for functions of the real line, by extending any component of $v$ to an even function in $H^1(\mathbb R)$. Similarly, one can prove that standard Sobolev inequalities hold on $\mathcal{G}_L$ in the same form as in the classical case.

\subsection{Hamiltonian operator}

Let $\gamma \in \mathbb R$. We equip $\mathcal{G}_L$ with the \emph{Hamiltonian operator with $\delta$-vertex condition}, i.e. the operator $\operatorname{H}_\gamma : L^2(\mathcal{G}_L) \to L^2(\mathcal{G}_L)$ defined by
\[
    \operatorname{H}_\gamma v = \left( -v_c'', -v_h'' \right)
\]
with domain
\[
    D \left( \operatorname{H}_\gamma \right) = \left\{ v \in H^2(\mathcal{G}_L) : v_c(L) = v_c(-L) = v_h(0) \text{ and } v_c'(L) - v_c'(-L) + v_h'(0) = \gamma v(\operatorname{v}) \right\}.
\]
This operator is self-adjoint (see e.g. \cite[Chapter I]{BeKu13}). Observe that the domain contains a continuity condition at the vertex and a jump condition for the derivatives and that for $\gamma = 0$, we recover the classical Kirchhoff conditions. The spectrum is given by the following proposition, which is proved in Appendix \ref{secSpectrum}.
\begin{proposition} \label{propSpectrum}
    The essential spectrum of $\operatorname{H}_\gamma$ is given by
    \begin{equation*}
        \sigma_{\operatorname{ess}} \left( \operatorname{H}_\gamma \right) = [0, \infty).
    \end{equation*}
    The discrete spectrum of $\operatorname{H}_\gamma$ is given by
    \begin{equation*}
        \sigma_{\operatorname{dis}} \left( \operatorname{H}_\gamma \right) =
        \begin{cases}
            \emptyset & \text{ if } \gamma \geq 0, \\
            \left\{ \lambda_\gamma \right\} & \text{ if } \gamma < 0,
        \end{cases}
    \end{equation*}
    where $\lambda_\gamma$ is the unique negative solution of
    \begin{equation} \label{eqNegEigVal}
        -\sqrt{|\lambda|} \left( 2 \tanh \left( \sqrt{|\lambda|}L \right) + 1 \right) = \gamma.
    \end{equation}
\end{proposition}

The quadratic form $Q_{\operatorname{H}_\gamma}$ of $\operatorname{H}_\gamma$ is given by
\[
    Q_{\operatorname{H}_\gamma}(v) = \|v'\|_{L^2(\mathcal{G}_L)}^2 + \gamma |v(\operatorname{v})|^2
\]
with domain
\[
    D \left( Q_{\operatorname{H}_\gamma} \right) = H_D^1(\mathcal{G}_L)
=
    H_D^1 \left( \mathcal{G}_L \right) = \left\{ v \in H^1(\mathcal{G}_L) : v_c(L) = v_c(-L) = v_h(0) \right\}.
\]

In this work, we will focus on the case of a \emph{repulsive} $\delta$-vertex condition. Hence $\gamma > 0$ will always be assumed in the sequel.

\subsection{Variational problem}

In this subsection, we explain the relation between equation \eqref{eqNLS} and the variational problem \eqref{eqMinAction}.

Solutions of \eqref{eqNLS} are critical points of the action $S_{\omega, \gamma, L}: H_D^1(\mathcal{G}_L) \to \mathbb R$ defined in \eqref{eqAction} and whose explicit expression is given by
\begin{equation*}
    S_{\omega, \gamma, L} (v) = \frac{1}{2} \|v'\|_{L^2(\mathcal{G}_L)}^2 + \frac{\gamma}{2} |v(\operatorname{v})|^2 + \frac{\omega}{2} \|v\|_{L^2(\mathcal{G}_L)}^2 - \frac{1}{4} \|v\|_{L^4(\mathcal{G}_L)}^4.
\end{equation*}
Indeed, a function $u \in H_D^1(\mathcal{G}_L)$ satisfies $S_{\omega, \gamma, L}' (u) = 0$ if and only if $u \in D(\operatorname{H}_\gamma)$ and is a solution of \eqref{eqNLS}. Observe that $S_{\omega, \gamma, L}$ is not bounded on $H_D^1(\mathcal{G}_L)$. Indeed, for $v \in H_D^1(\mathcal{G}_L)$, we have that
\[
    S_{\omega, \gamma, L}(tv) = \frac{t^2}{2} \|v'\|_{L^2(\mathcal{G}_L)}^2 + \frac{\gamma t^2}{2} |v(\operatorname{v})|^2 + \frac{\omega t^2}{2} \|v\|_{L^2(\mathcal{G}_L)}^2 - \frac{t^4}{4} \|v\|_{L^4(\mathcal{G}_L)}^4
    \to -\infty
\]
as $t \to \infty$.

The Nehari functional $I_{\omega, \gamma, L}: H_D^1(\mathcal{G}_L) \to \mathbb R$ defined in \eqref{eqNehariFunctional} is given by
\[
    I_{\omega, \gamma, L} (v) = \left. \frac{d}{dt} S_{\omega, \gamma, L}(t v) \right|_{t=1} = \left\langle S_{\omega, \gamma, L}' (v), v \right\rangle_{H_D^{1^\star}(\mathcal{G}_L), H_D^1(\mathcal{G}_L)},
\]
where we denoted by $H_D^{1^\star} (\mathcal{G}_L)$ the dual space of $H_D^1 (\mathcal{G}_L)$, so that for $v \in H_D^1(\mathcal{G}_L)$ we have
\begin{equation*}
    I_{\omega, \gamma, L} (v) = \|v'\|_{L^2(\mathcal{G}_L)}^2 + \gamma |v(\operatorname{v})|^2 + \omega \|v\|_{L^2(\mathcal{G}_L)}^2 - \|v\|_{L^4(\mathcal{G}_L)}^4.
\end{equation*}
 Observe that a non trivial critical point of $S_{\omega, \gamma, L}$ naturally belongs to the Nehari manifold $\mathcal N_{\omega, \gamma, L}$ defined in \eqref{eqNehariManifold}. Furthermore, for $v \in \mathcal N_{\omega, \gamma, L}$, we have
\begin{equation} \label{eqBoundAction}
    S_{\omega, \gamma, L} (v) = \frac{1}{4} \|v\|_{L^4(\mathcal{G}_L)}^4,
\end{equation}
so that $S_{\omega, \gamma, L}$ is bounded from below on $\mathcal N_{\omega, \gamma, L}$.

By \eqref{eqBoundAction}, an action ground state $u$ is also a solution of the variational problem
\begin{equation}
    s_{\omega, \gamma, L} = \inf \left\{ \frac{1}{4} \|v\|_{L^4(\mathcal{G}_L)}^4 : v \in \mathcal N_{\omega, \gamma, L} \right\}.
\end{equation}
Assume that $u\neq 0$ is an action ground state. Since $u$ is obtained as a minimizer, 
there exists a Lagrange multiplier $\lambda\in\R$ such that
\[
    S_{\omega, \gamma, L}' (u) = \lambda I_{\omega, \gamma, L}'(u),
\]
so that
\begin{align*}
    0 = I_{\omega, \gamma, L}(u)
    & = \left\langle S_{\omega, \gamma, L}' (u), u \right\rangle_{H_D^{1^\star}(\mathcal{G}_L), H^1(\mathcal{G}_L)} \\
    & = \lambda \left\langle I_{\omega, \gamma, L}' (u), u \right\rangle_{H_D^{1^\star}(\mathcal{G}_L), H^1(\mathcal{G}_L)}.
\end{align*}
However, we have 
\[
    \left\langle I_{\omega, \gamma, L}' (u), u \right\rangle_{H_D^{1^\star}(\mathcal{G}_L), H^1(\mathcal{G}_L)} = -2 \| u \|_{L^4(\mathcal{G}_L)}^4 \neq 0,
\]
from which we infer that $\lambda = 0$. Thus, $u$ is a solution of \eqref{eqNLS}.

\subsection{Phase portrait and period functions}

In this subsection, we introduce some elements related to the phase portrait of the ordinary differential equation verified by the stationary solution on each branch of the tadpole. This will be useful to study the shape of the action ground state.

Let $I = [-L, L] \text{ or } [0, \infty)$, $\omega > 0$ and $\phi \in H^2(I)$ be a solution of the equation
\begin{equation} \label{eqNLS1D}
    -\phi'' + \omega \phi - |\phi|^2 \phi = 0
\end{equation}
on $I$. Multiplying \eqref{eqNLS1D} by $\phi'$ and integrating, we obtain that there exists $E \in (-\omega^2 /2, \infty)$ such that for $x \in I$ we have
\begin{align*}
    -\phi'(x)^2 + \omega \phi(x)^2 - \frac{\phi(x)^4}{2} = -E.
\end{align*}
 We define the \emph{integral of motion} $\mathcal{E}_{\omega}$ for $(p,q) \in \mathbb R^2$ by
\begin{equation} \label{eqIntegralMotion}
    \mathcal{E}_{\omega} (p, q) = q^2 - p^2 \left( \omega - \frac{p^2}{2} \right)
\end{equation}
so that for $x \in I$ we have 
\[
    \mathcal{E}_{\omega} \left( \phi(x), \phi'(x) \right) = E.
\]
We define the \emph{$\mathcal{E}_{\omega}$-curve of level $E$} as the set
\begin{equation} \label{eqECurve}
    \mathcal{E}_{\omega, E} = \left\{ (p,q) \in \mathbb R^2 : \mathcal{E}_{\omega}(p,q) = E \right\}.
\end{equation}
We recall the definition of the complete elliptic integral of the first kind: for $k \in [0,1]$,
\begin{equation} \label{eqPeriod}
    K(k) = \int_0^{\frac{\pi}{2}} \frac{d\theta}{\sqrt{1 - k^2 \sin(\theta)^2}}.
\end{equation}
The nature of $\phi$ is different according to the different values of $\mathcal{E}_{\omega}(\phi, \phi')$. The explicit solution can be expressed in terms of a Jacobi elliptic function or a hyperbolic secant in the following way (see e.g. \cite{GuLeTs17}):
\begin{itemize}
    \item if $-\omega^2/2 \leq \mathcal{E}_{\omega}(\phi,\phi') < 0$, $\phi$ is a \emph{dnoidal-type} solution. There exists $k \in [0,1)$ and $a \in [-\sqrt{(1 - 2k^2) / \omega}K(k), \sqrt{(1 - 2k^2) / \omega}K(k))$ such that $\phi = \phi_{\operatorname{dn}, k, a}$, where $\phi_{\operatorname{dn},k,a}$ is given by
    \begin{equation} \label{eqSolDn}
        \phi_{\operatorname{dn}, k ,a}(x) = \sqrt{\frac{2\omega}{2-k^2}} \operatorname{dn} \left( \sqrt{\frac{\omega}{2-k^2}}x + a; k \right)
    \end{equation}
    for $x \in I$ and is periodic of period $2 \sqrt{(2 - k^2) / \omega}K(k)$; observe that $\phi_{\operatorname{dn}, k, a}$ is even if and only if $a = 0$ or $a = \sqrt{(1-2k^2)/\omega}K(k)$;
    \item if $\mathcal{E}_{\omega}(\phi,\phi') = 0$, $\phi$ is a hyperbolic secant. There exists $b \in \mathbb R$ such that $\phi = \phi_{\operatorname{sech}, b}$, where $\phi_{\operatorname{sech}, b}$ is given by
    \begin{equation} \label{eqSolSech}
        \phi_{\operatorname{sech}, b}(x) = \sqrt{2\omega} \operatorname{sech} \left( \sqrt{\omega}x + b \right),
    \end{equation}
    for $x \in I$; for $I = [0, \infty)$, we say that $u_c$ is a \emph{half soliton-type} solution if $b = 0$, a \emph{tail-type} solution if $b > 0$ and a \emph{bump-type} solution if $b < 0$;
    \item if $\mathcal{E}_{\omega}(\phi,\phi') > 0$, $\phi$ is a \emph{cnoidal-type} solution. There exists $k \in (1/\sqrt{2},1)$ and $a \in [-2\sqrt{(1 - 2k^2) / \omega}K(k), 2\sqrt{(1 - 2k^2) / \omega}K(k))$ such that $\phi = \phi_{\operatorname{cn}, k, a}$, where $\phi_{\operatorname{cn}, k, a}$ is given by
    \begin{equation} \label{eqSolCn}
        \phi_{\operatorname{cn}, k, a}(x) = \sqrt{\frac{2\omega k^2}{2k^2-1}} \operatorname{cn} \left(\sqrt{\frac{\omega}{2k^2-1}} x + a; k \right),
    \end{equation}
    for $x \in I$ and is periodic of period $4\sqrt{(1 - 2k^2) / \omega}K(k)$; observe that $\phi_{\operatorname{cn}, k, a}$ is even if and only if $a = 0$ or $a = 2\sqrt{(1 - 2k^2) / \omega}K(k)$.
\end{itemize}
See \cite{ByFr71} for a comprehensive study of Jacobi elliptic functions.

Let $p_0 \in [0, \sqrt{2\omega}]$. For a later use, we define $\mathcal E_{\omega, E}^{p_0, 1}$, $\mathcal E_{\omega, E}^{p_0, 2}$ and $\mathcal E_{\omega, E}^{p_0, 3}$ as the subsets of the $\mathcal E$-curve of level $E$, defined in \eqref{eqECurve}, given by
\begin{equation} \label{eqECurveSubset}
    \begin{aligned}
        \mathcal E_{\omega, E}^{p_0, 1} & = \left\{ (p,q) \in \mathcal E_{\omega, E}: p \leq p_0, q \leq 0  \right\}, \\
        \mathcal E_{\omega, E}^{p_0, 2} & = \left\{ (p,q) \in \mathcal E_{\omega, E}: p \geq p_0  \right\}, \\
        \mathcal E_{\omega, E}^{p_0, 3} & = \left\{ (p,q) \in \mathcal E_{\omega, E}: p \leq p_0, q \geq 0  \right\}.
    \end{aligned}
\end{equation}

\begin{remark} \label{rkSech}
    If $I = [0, \infty)$, then $\mathcal{E}_{\omega}(\phi,\phi') = 0$ and $\phi$ is a hyperbolic secant of type \eqref{eqSolSech} in order for $\phi$ to belong to $H^2(I)$, as $\phi$ decays to $0$. Furthermore, by the sense of the flow on the phase portrait, we have:
    \begin{itemize}
        \item $\phi$ is a tail-type solution if and only if  $ \phi'(0)<0 $, in which case
        \[
            \left\{ \left( \phi(x), \phi'(x) \right) : x \in [0, \infty) \right\} = \mathcal E_{\omega, 0}^{\phi(0), 1};
        \]
        \item $\phi$ is a half soliton-type solution if and only if $\left( \phi(0), \phi'(0) \right) = \left( \sqrt{2\omega}, 0 \right)$, in which case
        \[
            \{(\phi(x), \phi'(x)) : x \in [0, \infty) \} = \mathcal E_{\omega, 0}^{\sqrt{2\omega}, 1};
        \]
        \item $\phi$ is a bump-type solution if and only if  $ \phi'(0)>0$, in which case
        \[
            \left\{ \left( \phi(x), \phi'(x) \right) : x \in [0, \infty) \right\} = \mathcal E_{\omega, 0}^{\phi(0), 1} \cup \mathcal E_{\omega, 0}^{\phi(0), 2}.
        \]
    \end{itemize}
\end{remark}

Figure \ref{figECurve} illustrates the phase portrait of the cubic nonlinear Schrödinger equation.

\begin{figure}[ht]
    \centering
    \includegraphics[width=0.6\linewidth]{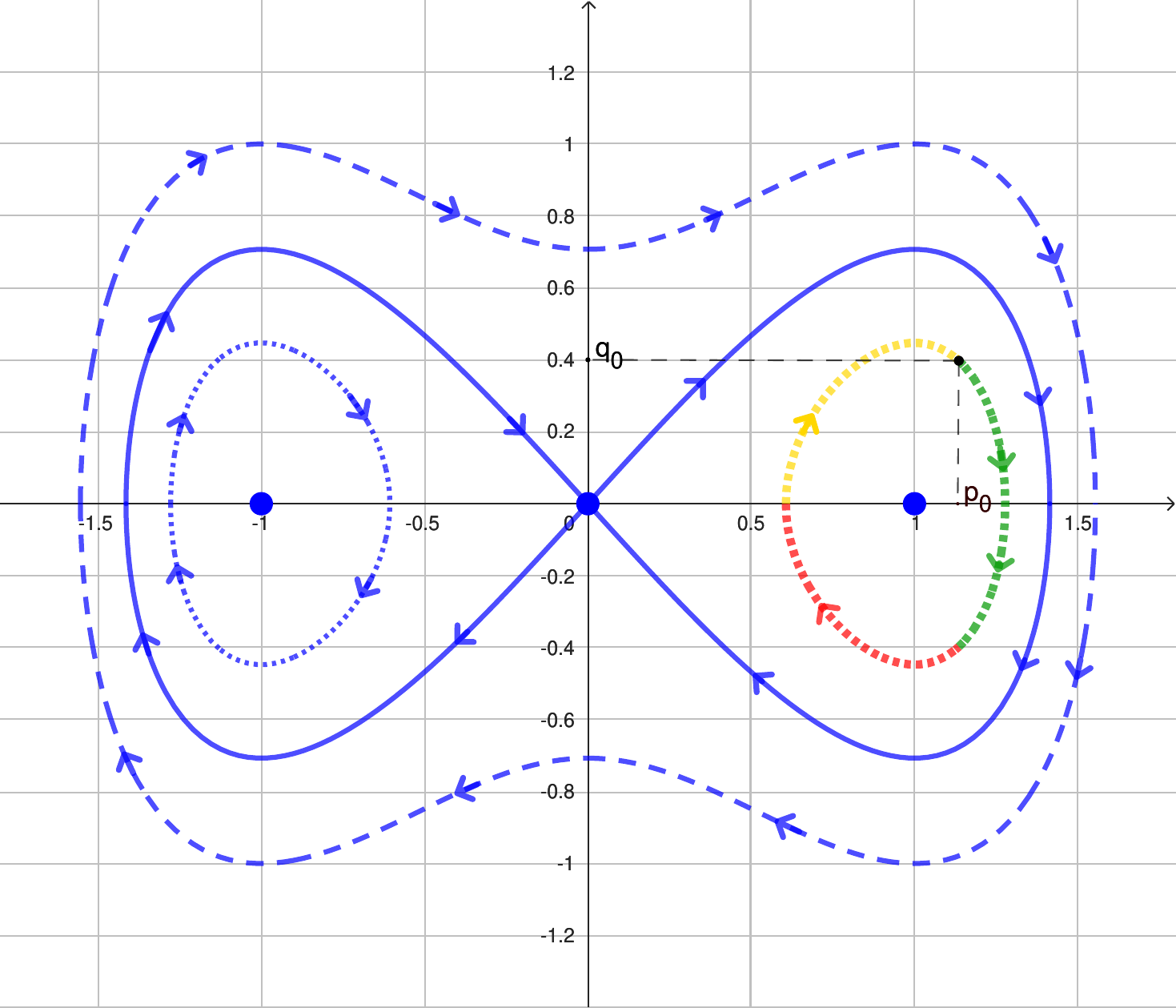}
    \caption{Phase portrait of the cubic nonlinear Schrödinger equation. The small dashed-curve is $\mathcal E_{\omega, -0.3}$ (dnoidal-type solution), the continuous curve is $\mathcal E_{\omega, 0}$ (hyperbolic secant), the large dashed-curve is $\mathcal E_{\omega, 0.5}$ (cnoidal-type solution) and the bold blue points are the fixed points, corresponding to $(-\sqrt{\omega},0)$, $(0,0)$ and $(\sqrt{\omega},0)$; the arrows represent the flow of the equation; for $E = -0.3$ and $(p_0, q_0) \in \mathcal E_{\omega, E}$, the sets $\mathcal E_{\omega, -0.3}^{p_0, 1}$ (red), $\mathcal E_{\omega, -0.3}^{p_0, 2}$ (green) and $\mathcal E_{\omega, -0.3}^{p_0, 3}$ (yellow) have been represented.}
    \label{figECurve}
\end{figure}

For $(p,q) \in \mathbb R_+ \times \mathbb R$, we define $p_+(p,q)$ and $p_-(p,q)$ as the solutions of the equation
\[
    \mathcal{E}_{\omega} (X,0) = - X^2 \left( \omega - \frac{X^2}{2} \right) = \mathcal{E}_{\omega} (p,q),
\]
given by
\begin{equation*}
    p_+ (p,q) = \sqrt{\omega \left( 1 + \sqrt{1 + \frac{2 \mathcal{E}_{\omega}(p,q)}{\omega^2}} \right)}
\end{equation*}
and
\begin{equation*}
    p_- (p,q) =
    \begin{cases}
        \sqrt{\omega \left( 1 - \sqrt{1 + \frac{2 \mathcal{E}_{\omega}(p,q)}{\omega^2}} \right)} & \text{ if } \mathcal{E}_{\omega} (p,q) \in (-\omega^2 /2, 0], \\
        - \sqrt{\omega \left( 1 + \sqrt{1 + \frac{2 \mathcal{E}_{\omega}(p,q)}{\omega^2}} \right)} & \text{ if } \mathcal{E}_{\omega} (p,q) \in (0, \infty).
    \end{cases}
\end{equation*}
In particular, they satisfy
\begin{gather*}
    p_-(p,q) \leq p \leq p_+(p,q), \quad
    p_+(0,0) = \sqrt{2\omega}, \quad
    p_-(0,0) = 0.
\end{gather*}
\begin{remark} \label{rkPeriodL}
    Let $\phi$ be a positive and even solution of \eqref{eqNLS1D} on $[-L,L]$. Then,
    \[
        \phi(0) = p_+(\phi(-L),\phi'(-L)) \text{ and } \phi'(0) = 0.
    \]    
\end{remark}
We define the \emph{period functions} $\mathcal T_+$ and $\mathcal T_-$ by
\begin{equation} \label{eqPeriodFunctions}
    \begin{aligned}
        \mathcal T_+ (p, q) & = \int_p^{p_+(p,q)} \frac{ds}{\sqrt{\mathcal{E}_{\omega} (p,q) + s^2 \left( \omega - \frac{s^2}{2} \right)}} & \text{ for } (p,q) & \in (\mathbb R_+ \times \mathbb R) \setminus \{ (0,0) \}, \\
        \mathcal T_- (p, q) & = \int_{p_-(p,q)}^{p} \frac{ds}{\sqrt{\mathcal{E}_{\omega} (p,q) + s^2 \left( \omega - \frac{s^2}{2} \right)}} & \text{ for } (p,q) & \in (\mathbb R_+ \times \mathbb R) \setminus \mathcal{E}_{\omega, 0}.
    \end{aligned}
\end{equation}
For $x \in I$, the period function $\mathcal T_+$ (respectively $\mathcal T_-$) measures the distance (in $x$) for a solution  of \eqref{eqNLS1D} between the point $(p, q)$ and the point $(p_+(p, q), 0)$ (respectively $(p_-(p, q), 0)$) obtained along the $\mathcal E_\omega$-curve of level $\mathcal E_\omega(p, q)$. In particular, if $T \in (0, \infty]$ is the period of the solution $\phi$ of \eqref{eqNLS1D}, for $x \in I$ we have
\begin{equation} \label{eqPeriodSolution}
    T = 2 \left( \mathcal T_+ \left( \phi(x), \phi'(x) \right) + \mathcal T_- \left( \phi(x), \phi'(x) \right) \right).
\end{equation}
The period functions have been introduced in the context of metric graphs in \cite{NoPe20}, \cite{KaPe21} and \cite{KaNoPe22}. They have been recently the object of an in depth study in \cite{AgCoTa24,AgCoTa25}.

The following lemma gives some elementary properties of the period functions. 
\begin{lemma} \label{lemPeriodFunc}
    The functions $\mathcal T_+$ and $\mathcal T_-$ are differentiable. Furthermore, $\mathcal T_+$ satisfies
    \begin{align}
        & \mathcal T_+(p,q) > 0 & \text{ for } (p, q) & \in \left( \left( 0, \sqrt{2\omega} \right) \times \mathbb R^* \right) \cup \left( \left( 0, \sqrt{\omega} \right) \times \{ 0 \} \right), \label{eqPeriod1} \\
        & \mathcal T_+(p,q) = 0 & \text{ for } (p, q) & \in \left[ \sqrt{\omega}, \sqrt{2\omega} \right] \times \{ 0 \}, \label{eqPeriod2} \\
        & \mathcal T_+(p,q) \to \infty & \text{ as } (p, q) & \to (0, 0), \label{eqPeriod3}
    \end{align}
    and $\mathcal T_-$ satisfies
    \begin{align}
        & \mathcal T_-(p,q) > 0 & \text{ for } (p, q) & \in \left( \left[ 0, \sqrt{2\omega} \right) \times \mathbb R^* \right) \cup \left( \left( \sqrt{\omega}, \sqrt{2 \omega} \right) \times \{ 0 \} \right), \label{eqPeriod4} \\
        & \mathcal T_-(p,q) = 0 & \text{ for } (p, q) & \in \left[ 0, \sqrt{\omega} \right] \times \{ 0 \}, \label{eqPeriod5} \\
        & \mathcal T_-(p,q) \to \infty & \text{ as } (p, q) & \to \left( p^*, q^* \right) \label{eqPeriod6}
    \end{align}
    for $(p^*,q^*) \in \mathcal{E}_{\omega, 0}$.
\end{lemma}

\begin{proof}
    The functions $\mathcal T_+$ and $\mathcal T_-$ are differentiable as the functions $\mathcal{E}_{\omega}$, $p_+$ and $p_-$ are differentiable. Furthermore, we have
    \begin{align*}
        (p, q) \in \left[ \sqrt{\omega}, \sqrt{2\omega} \right] \times \{ 0 \} & \text{ if and only if } p = p_+(p, q), \\
        (p, q) \in \left[ 0, \sqrt{\omega} \right] \times \{ 0 \} & \text{ if and only if } p = p_-(p, q).
    \end{align*}
    This proves \eqref{eqPeriod2} and \eqref{eqPeriod4}. As the integrand is positive, \eqref{eqPeriod1} and \eqref{eqPeriod3} hold. Finally, by Fatou's lemma, we have
    \begin{align*}
        \liminf_{(p,q) \to 0} \mathcal T_+(p,q) & \geq \int_0^{\sqrt{2\omega}} \frac{ds}{\sqrt{s^2 \left( \omega - \frac{s^2}{2} \right)}} = \infty, \\
        \liminf_{(p,q) \to (p^*,q^*)} \mathcal T_-(p,q) & \geq \int_0^{p^*} \frac{ds}{\sqrt{s^2 \left( \omega - \frac{s^2}{2} \right)}} = \infty.
    \end{align*}
    This proves \eqref{eqPeriod3} and \eqref{eqPeriod6} and concludes the proof.
\end{proof}

\section{Analysis of ground states}
\label{sec:ground_states}

We first give a criterion (Lemma \ref{lemExistenceCriterion})  for the existence of action ground states. This criterion is standard in the case of energy ground states (see e.g. \cite{AdCaFiNo14,AdSeTi15b,AdSeTi16}) and, as we will see, can also be adapted to the case of action ground states (see \cite{DeDoGaSe23} for the relations between action ground states and energy ground states). 
We then deduce from the criterion an existence result for given sets of parameters $\omega$, $\gamma$ and $L$. %

The proof of the existence of action ground states is based on the comparison between the infimum  $s_{\omega, \gamma, L}$  of the action restrained on the Nehari manifold on the graph (defined in \eqref{eqMinAction}) and the corresponding one on the line. We define the action on the line $S_{\omega,\mathbb R}: H^1(\mathbb R) \to \mathbb R$ by
\begin{equation} \label{eqActionLine}
    S_{\omega, \mathbb R}(v) = \frac{1}{2} \|v'\|_{L^2(\mathbb R)}^2 + \frac{\omega}{2} \|v\|_{L^2(\mathbb R)}^2  - \frac{1}{4} \|v\|_{L^4(\mathbb R)}^4,
\end{equation}
and the Nehari functional on the line $I_{\omega,\mathbb R}: H^1(\mathbb R) \to \mathbb R$ by
\begin{equation} \label{eqNehariLine}
    I_{\omega, \mathbb R}(v) = \|v'\|_{L^2(\mathbb R)}^2 + \omega \|v\|_{L^2(\mathbb R)}^2 - \|v\|_{L^4(\mathbb R)}^4 .
\end{equation}
We denote by $s_{\omega, \mathbb R}$ the action level of the ground state on the Nehari manifold on $\mathbb R$, i.e.
\begin{equation} \label{eqMinActionLine1}
    s_{\omega, \mathbb R} = \inf \left\{ S_{\omega, \mathbb R}(v) : v \in H^1(\mathbb R)\setminus\{0\},\; I_{\omega, \mathbb R}(v) = 0 \right\}.
\end{equation}
It is known (see \cite{Ca03}) that the ground state on $\mathbb R$ is given, up to translation, for $x \in \mathbb R$, by
\begin{equation} \label{eqGSLine}
    \phi(x) = \sqrt{2\omega} \operatorname{sech} \left( \sqrt{\omega} x \right),
\end{equation}
and a direct computation shows that
\begin{equation} \label{eqMinActionLine2}
    s_{\omega, \mathbb R} = \frac{4}{3} \omega^\frac{3}{2}.
\end{equation}

The following lemma gives a criterion for the existence of an action ground state.

\begin{lemma} \label{lemExistenceCriterion}
    If the condition
    \[
        s_{\omega, \gamma, L} < \frac{4}{3} \omega^\frac{3}{2},
    \]
    holds, then the variational problem \eqref{eqMinAction} admits a minimizer. %
\end{lemma}

The Lemma will be proved using a profile decomposition argument the Palais-Smale sequences. 

Recall that a \emph{Palais-Smale sequence for $S_{\omega, \gamma, L}$ at level $c$}, denoted as a $(PS)_c$ sequence, is a sequence $(v_n)_{n \in \mathbb N}$ such that
\[
    S_{\omega, \gamma, L} (v_n) \to c \text{ and } S_{\omega, \gamma, L}' (v_n) \to 0 \text{ in } H^{-1}(\mathcal{G}_L) \text{ as } n \to \infty.
\]
The profile decomposition is a refinement of the concentration-compactness principle. In the present context, it takes the form of the following lemma, which is in the spirit of the work of \cite{JeKa08} and could be proved following similar arguments. 

\begin{lemma} \label{lemProfileDecomposition}
    Let $(v_n = (v_{n,c}, v_{n,h}))_{n \in \mathbb N}$ be a bounded $(PS)_{s_{\omega, \gamma, L}}$ sequence for $S_{\omega, \gamma, L}$. Then there exist $v$ solution of \eqref{eqNLS}; an integer $k \geq 0$; for $i = 1,\dots,k$, sequences $(x_n^i)_{n \in \mathbb N} \subset \mathbb R^+$ and non-zero functions $\phi_i \in H^1(\mathbb R)$ satisfying $S_{\omega, \mathbb R}'(\phi_i) = 0$; such that $(v_n)_{n \in \mathbb N}$ satisfies up to a subsequence:
    \begin{align*}
        & v_n \rightharpoonup v \text{ weakly in } H^1(\mathcal{G}_L), \\
        & S_{\omega, \gamma, L} (v_n) \to s_{\omega, \gamma, L} = S_{\omega, \gamma, L}(v) + \sum_{i = 0}^k S_{\omega, \mathbb R}(\phi_i), \\
        & v_{n, c} - v_c \to 0 \text{ strongly in } H^1(-L,L), \\
        & v_{n, h} - \left(v_h + \sum_{i = 0}^k \phi_i (\cdot - x_n^i) \right) \to 0 \text{ strongly in } H^1(0,\infty), \\
        & x_n^i \to \infty \text{ and } |x_n^i - x_n^j| \to \infty \text{ for } 1 \leq i \neq j \leq k,
    \end{align*}
    as $n \to \infty$, where we agree that in the case $k = 0$, the above holds without $\phi_i$ and $x_n^i$.
\end{lemma}
We may now proceed to the proof of the existence criterion.

\begin{proof}[Proof of Lemma \ref{lemExistenceCriterion}]
    Let $(u_n)_{n \in \mathbb N} \subset \mathcal N_{\omega, \gamma, L}$ be a minimizing sequence for $S_{\omega, \gamma, L}$. It is a $(PS)_{s_{\omega, \gamma, L}}$ sequence, as $S_{\omega, \gamma, L}(u_n) \to s_{\omega, \gamma, L}$  and $S_{\omega, \gamma, L}'(u_n) \to 0$ as $n \to \infty$. By contraposition, assume that $(u_n)_{n \in \mathbb N}$ does not converge. Then the case $k = 0$ of Lemma \ref{lemProfileDecomposition} cannot occurs. Thus, there exists $u$ solution of \eqref{eqNLS} and $\phi_1 \in H^1(\mathbb R)$ satisfying $S_{\omega, \mathbb R}'(\phi_1) = 0$ such that
    \begin{align*}
        s_{\omega, \gamma, L} & = \liminf S_{\omega, \gamma, L} (u_n) \\
        & \geq S_{\omega, \gamma, L} (u) + S_{\omega, \mathbb R} (\phi_1) \\
        & \geq S_{\omega, \gamma, L} (u) + s_{\omega, \mathbb R}.
    \end{align*}
    This concludes the proof.
\end{proof}

We reformulate the variational problem \eqref{eqMinAction} in a more convenient way. 

\begin{lemma} \label{lemEquivalentPB}
    The variational problem
    \begin{equation} \label{eqMinAction2}
        \inf \left\{ \frac{1}{4} \| v \|_{L^4(\mathcal{G}_L)}^4 : v \in H_D^1(\mathcal{G}_L) \text{ and } I_{\omega, \gamma, L}(v) \leq 0 \right\}.
    \end{equation}
    is equivalent to \eqref{eqMinAction}. Furthermore, if a minimizer exists, then it can be chosen to be non-negative.
\end{lemma}

\begin{proof}
    We first observe that from \eqref{eqBoundAction} we have
    \begin{equation} \label{eqMinActionGeq}
        s_{\omega, \gamma, L} \geq \inf \left\{ \frac{1}{4} \| v \|_{L^4(\mathcal{G}_L)}^4 : v \in H_D^1(\mathcal{G}_L) \text{ and } I_{\omega, \gamma, L}(v) \leq 0 \right\}
    \end{equation}
     For $v \in H_D^1(\mathcal{G}_L)$ such that $I_{\omega, \gamma, L}(v) \leq 0$, we have
    \[
        I_{\omega, \gamma, L}(t v) = t^2 \left( \|v'\|_{L^2(\mathcal{G}_L)}^2 + \omega \|v\|_{L^2(\mathcal{G}_L)}^2 + \gamma |v(0)|^2 \right) - t^4 \|v\|_{L^4(\mathcal{G}_L)}^4
    \]
    for $t \geq 0$. Therefore, there exists $0 \leq t^* < 1$ such that $I_{\omega, \gamma, L}(t^* v) = 0$ and
    \begin{align*}
        s_{\omega, \gamma, L} \leq S_{\omega, \gamma, L}(t^* v)
        = \frac{(t^*)^4}{4} \| v \|_{L^4(\mathcal{G}_L)}^4
        < \frac{1}{4} \| v \|_{L^4(\mathcal{G}_L)}^4 %
    \end{align*}
    so that
    \begin{equation} \label{eqMinActionLeq}
        s_{\omega, \gamma, L} \leq \inf \left\{ \frac{1}{4} \| v \|_{L^4(\mathcal{G}_L)}^4 : v \in H_D^1(\mathcal{G}_L) \text{ and } I_{\omega, \gamma, L}(v) \leq 0 \right\}.
    \end{equation}
    We obtain
    \begin{equation*} 
        s_{\omega, \gamma, L} = \inf \left\{ \frac{1}{4} \| v \|_{L^4(\mathcal{G}_L)}^4 : v \in H_D^1(\mathcal{G}_L) \text{ and } I_{\omega, \gamma, L}(v) \leq 0 \right\}
    \end{equation*}
    by combining \eqref{eqMinActionGeq} and \eqref{eqMinActionLeq}. Finally, let $u$ be a solution to the variational problem \eqref{eqMinAction}. Observe that
    \[
        I_{\omega, \gamma, L} \left( |u| \right) \leq I_{\omega, \gamma, L} \left( u \right), \quad \left\| |u| \right\|_{L^4(\mathcal G)}^4 = \| u \|_{L^4(\mathcal G)}^4,
    \]
    Hence $|u|$, which is nonnegative, is also a minimizer. This concludes the proof.
\end{proof}

We now state the following proposition, giving the existence of an action ground state for given sets of parameters $\omega$, $\gamma$ and $L$. The proof is based on the computation of the Nehari functional and the action of a suitable competitor in order to satisfy the existence criterion given by Lemma \ref{lemEquivalentPB}.

\begin{proposition} \label{propExistence}
    The following assertions hold:
    \begin{enumerate}[label=(\roman*)]
        \item For $\omega > 0$, $0 < \gamma < ( 1 - \sqrt{3}/2 )^{1/3} \sqrt{\omega}$ and $L < L^*$, where $L^* = L^*(\omega, \gamma)$ is given by
        \begin{equation} \label{eqSmallL}
            L^* = \frac{\left( 2 - \sqrt{3} \right) \omega^{\frac{3}{2}} - 2 \gamma^3}{6 \omega^2},
        \end{equation}
        the variational problem \eqref{eqMinAction} admits a minimizer.
        \item For $\omega > 0$, $0 < \gamma < \sqrt{\omega}$ and $L \geq L^{**}$, where $L^{**} = L^{**}(\omega, \gamma)$ is given by
        \begin{equation} \label{eqBigL}
            L^{**} = \frac{1}{\sqrt{\omega}} \operatorname{arctanh} \left( \frac{\gamma}{\sqrt{\omega}} \right),
        \end{equation}
        the variational problem \eqref{eqMinAction} admits a minimizer.
    \end{enumerate}
\end{proposition} 

\begin{proof}
    By Lemma \ref{lemExistenceCriterion}, Lemma \ref{lemEquivalentPB} and \eqref{eqMinActionLine2}, it is enough to find a candidate $v \in H_D^1(\mathcal{G}_L)$ which satisfies the properties
    \begin{equation} \label{eqConditionNF}
        I_{\omega, \gamma, L}(v) \leq 0
    \end{equation}
    and
    \begin{equation} \label{eqConditionAc}
        \frac{1}{4} \| v \|_{L(\mathcal{G})}^4 < \frac{4}{3} \omega^\frac{3}{2}.
    \end{equation}
    
    In order to prove $(i)$, we define $v= (v_c, v_h) \in H_D^1(\mathcal{G}_L)$ by
    \begin{equation} \label{eqCandidate1}
        \begin{aligned}
            v_c (x) & = \sqrt{2(\omega - \gamma^2)} & \text{ for } x & \in [-L, L], \\
            v_h (x) & = \sqrt{2\omega} \operatorname{sech} \left( \sqrt{\omega} x - \operatorname{arctanh} \left( \frac{\gamma}{\sqrt{\omega}}\right) \right) & \text{ for } x & \in [0, \infty).
        \end{aligned}
    \end{equation}
    Observe that, in order for $v$ to be well-defined, the condition
    \begin{equation} \label{eqCondition1}
        \gamma^2 < \omega
    \end{equation}
    must hold. Furthermore, observe that $v_h$ solve the nonlinear Schrödinger equation with Robin-boundary condition
    \[
        v_h'(0) = \gamma v_h(0),
    \]
    in particular
    \[
        \omega \| v_h \|_{L^2(\mathbb R_+)}^2 + \| v_h' \|_{L^2(\mathbb R_+)}^2 - \| v_h \|_{L^4(\mathbb R_+)}^4 + \gamma \left| v_h(0) \right|^2 = 0.
    \]
    Then,
    \begin{equation*}
        \begin{aligned}
            I_{\omega, \gamma, L} (v) & = \omega \| v_c \|_{L^2(-L,L)}^2 + \| v_c' \|_{L^2(-L,L)}^2 - \| v_c \|_{L^4(-L,L)}^4 \\
            & = -4L \left( \omega - \gamma^2 \right) \left( \omega-2\gamma^2 \right),
        \end{aligned}
    \end{equation*}
    so that \eqref{eqConditionNF} holds if and only if
    \begin{equation} \label{eqCondition2}
        \gamma \leq \sqrt{\frac{\omega}{2}}.
    \end{equation}
    Furthermore,
    \[
        \frac{1}{4} \| v \|_{L(\mathcal{G})}^4 = 
        2L \omega^2 \left( 1 - \frac{\gamma^2}{\omega} \right)^2 + \left( \left( \frac{2}{3} + \frac{\gamma}{\sqrt{\omega}} \left( 1 - \frac{\gamma^2}{\omega} \right) \right)\omega^\frac{3}{2}+\frac{2\gamma^3}{3} \right).
    \]
    By \eqref{eqCondition1}, we have
    \begin{equation*}
        \frac{\gamma}{\sqrt{\omega}} \left( 1 - \frac{\gamma^2}{\omega} \right) \leq \frac{1}{\sqrt{3}} < \frac{2}{3}, \quad \left( 1 - \frac{\gamma^2}{\omega} \right)^2 \leq 1,
    \end{equation*}
    so
    \[
        \frac{1}{4} \| v \|_{L^4(\mathcal G)}^4 \leq 2L \omega^2 \left( 1 - \frac{\gamma^2}{\omega} \right)^2 +\frac{2 + \sqrt{3}}{3} \omega^{\frac{3}{2}}+ \frac{2}{3} \gamma^3.
    \]
    Thus, if
    \begin{equation} \label{eqCondition3bis}
        2L \omega^2 + \frac{2}{3} \gamma^3 < \frac{2 - \sqrt{3}}{3} \omega^{\frac{3}{2}}
    \end{equation}
    holds, then \eqref{eqConditionAc} holds. The condition \eqref{eqCondition3bis} is equivalent to the condition
    \begin{equation} \label{eqCondition3}
        L < \frac{\left( 2 - \sqrt{3} \right) \omega^{\frac{3}{2}} - 2 \gamma^3}{6 \omega^2}.
    \end{equation}
    Observe that \eqref{eqCondition3} makes sense if and only if the right-hand side term is positive, so that if and only if
    \[
        \left( 2 - \sqrt{3} \right) \omega^{\frac{3}{2}} - 2 \gamma^3 > 0,
    \]
    or equivalently
    \begin{equation} \label{eqCondition4}
        \gamma < \left( \frac{2 - \sqrt{3}}{2} \right)^{\frac{1}{3}} \sqrt{\omega}.
    \end{equation}
    By \eqref{eqCondition4} and \eqref{eqCondition3}, $v$ satisfies \eqref{eqConditionNF} and \eqref{eqConditionAc}. This concludes the proof of $(i)$.
    
    In order to prove $(ii)$, we define $v= (v_c, v_h) \in H_D^1(\mathcal{G}_L)$ by
    \begin{equation} \label{eqCandidate2}
        \begin{aligned}
            v_c (x) & = \sqrt{2\omega} \operatorname{sech} \left( \sqrt{\omega} x \right) & \text{ for } x & \in [-L, L], \\
            v_h (x) & = \sqrt{2\omega} \operatorname{sech} \left( \sqrt{\omega} \left( x + L \right) \right) & \text{ for } x & \in [0, \infty).
        \end{aligned}
    \end{equation}
    Then,
    \[
        I_{\omega, \gamma, L} (v) = 2\gamma \left( 1 - \operatorname{tanh} \left( \sqrt{\omega} L \right)^2 \right) \omega - 2 \operatorname{tanh} \left( \sqrt{\omega} L \right) \left( 1 - \operatorname{tanh} \left( \sqrt{\omega} L \right)^2 \right) \omega^\frac{3}{2}
    \]
    so that $I_{\omega, \gamma, L} (v) \leq 0$ if and only if
    \[
        \gamma \leq \sqrt{\omega} \operatorname{tanh} \left( \sqrt{\omega} L \right),
    \]
    which is implied by the conditions
    \begin{equation} \label{eqCondition5}
        L \geq \frac{1}{\sqrt{\omega}} \operatorname{arctanh} \left( \frac{\gamma}{\sqrt{\omega}} \right)
    \end{equation}
    and
    \begin{equation} \label{eqCondition6}
        \gamma < \sqrt{\omega}.
    \end{equation}
    Finally, we have that
    \[
        \frac{1}{4} \| v \|_{L^4(\mathcal{G}_L)}^4 = \int_{-L}^\infty \phi(x)^4 dx < \int_{-\infty}^\infty \phi(x)^4 dx = \frac{4}{3} \omega^\frac{3}{2},
    \]
    where $\phi$ has been defined in \eqref{eqGSLine}. By \eqref{eqCondition5} and \eqref{eqCondition6}, $v$ satisfies \eqref{eqConditionNF} and \eqref{eqConditionAc}. This concludes the proof of $(ii)$.
\end{proof}

\begin{remark}
    As the length of the compact edge shrinks to $0$, we can suppose that the ground state on $\mathcal G$ will be similar to the ground state on the half-line with the same $\delta$ condition. Thus, the candidate $v$ given by \eqref{eqCandidate1} has been chosen because it is similar to the ground state on the half-line with Robin-boundary condition at $0$, as $v_h'(0) = \gamma v_h(0)$.
    
    As the length of the compact edge goes to $\infty$, we can suppose that the ground state on $\mathcal G$ will be close to the hyperbolic secant on the compact edge and a tail of hyperbolic secant on the half-line. This justifies the choice of the candidate $v$ given by \eqref{eqCandidate2} for large length.
\end{remark}

\section{Analysis of non-negative stationary states}
\label{sec:non_negative}

In this section, we study the shape of the ground states of \eqref{eqMinAction} under the conditions of Proposition \ref{propExistence}. We start by studying the shape of solutions of \eqref{eqNLS} which are even on the compact edge. We show that such solutions do exist in a finite number and study their shape by a careful analysis of the phase portrait of stationary nonlinear Schrödinger equation. We then study the solutions which are not even on the compact edge. Finally, we prove Theorem \ref{thmMain}.

As Lemma \ref{lemEquivalentPB} holds, we can restrict ourselves to the study of nonnegative stationary states to determine the shape of the ground states. This is why, in what follows, a solution $u = (u_c, u_h)$ of \eqref{eqNLS} is always supposed to be non-negative. In particular, we have
\[
    u_c(-L) = u_c(L) = u_h(0) \geq 0.
\]

\subsection{Even stationary states}

In this subsection, we study the shape of the positive solutions of \eqref{eqNLS} which are even on the compact edge $c$. We define $q_{\omega, \gamma, +}$ and $q_{\omega, \gamma, -}$ for $p \in [0, \sqrt{2\omega}]$ by
\[
    q_{\omega, \gamma, +}(p) = \frac{p}{2} \left( \gamma + \sqrt{\omega - \frac{p^2}{2}} \right), \quad q_{\omega, \gamma, -}(p) = \frac{p}{2} \left( \gamma - \sqrt{\omega - \frac{p^2}{2}} \right).
\]
For $\gamma > 0$, we define the \emph{$\Gamma_{\omega}$-curve of level $\gamma$} as the set
\begin{equation} \label{eqGammaCurve}
    \Gamma_{\omega, \gamma} = \left\{ \left( p, q_{\omega, \gamma, +}(p) \right) : p \in \left( 0, \sqrt{2\omega} \right] \right\} \cup \left\{ \left( p, q_{\omega, \gamma, -}(p) \right) : p \in \left( 0, \sqrt{2\omega} \right] \right\}.
\end{equation}

We will show in Lemma \ref{lemShapeLine} that the $\delta$-vertex condition imposes to a solution $u = (u_c, u_h)$ of \eqref{eqNLS} such that $u_c$ is even on the compact edge $c$ to satisfy $(u_c(-L), u_c'(-L)) \in \Gamma_{\omega, \gamma}$. We first state the following lemma in order to define subsets of $\Gamma_{\omega, \gamma}$ useful for later.

\begin{lemma} \label{lemCurvePos}
    Let $\omega > 0$ and $0 < \gamma < \sqrt{\omega}$ be fixed. Then
    \begin{equation} \label{eqSignQ}
        \begin{aligned}
            & q_{\omega, \gamma, -} \left( 0 \right) = 0, \quad q_{\omega, \gamma, -} \left( \sqrt{2 \left( \omega - \gamma^2 \right)} \right) = 0, \\
            & q_{\omega, \gamma, -} (p) < 0 \text{ for } p \in \left( 0, \sqrt{2(\omega - \gamma^2)} \right], \\
            & q_{\omega, \gamma, -} (p) > 0 \text{ for } p  \in \left( \sqrt{2(\omega - \gamma^2)}, \sqrt{2\omega} \right].
        \end{aligned}
    \end{equation}
\end{lemma}

\begin{proof}
    We have
    \[
        q_{\omega, \gamma, -}'(p) = \frac{1}{2} \gamma - \frac{1}{2} \sqrt{\omega - \frac{p^2}{2}} + \frac{p^2}{4\sqrt{\omega - \frac{p^2}{2}}}
    \]
    for $p \in [0, \sqrt{2\omega})$, in particular, as $\gamma < \sqrt{\omega}$,  we get
    \begin{align}
        & q_{\omega, \gamma, -}'(0) = \frac{1}{2} \left(\gamma-\sqrt{\omega} \right) < 0, \label{eqVarQ1} \\
        & q_{\omega, \gamma, -}'\left( \sqrt{2(\omega - \gamma^2)} \right) = \frac{1}{2} \frac{\omega - \gamma^2}{\gamma}  > 0. \label{eqVarQ2}
    \end{align}
Furthermore, observe that
    \(
        q_{\omega, \gamma, -}(p) = 0
    \)
    if and only if $p = 0$ or $p = \sqrt{2(\omega - \gamma^2)}$. Thus, we get that $q_{\omega, \gamma, -} (p) < 0$ for $p \in ( 0, \sqrt{2(\omega - \gamma^2)} ]$ by \eqref{eqVarQ1} and that $q_{\omega, \gamma, -} (p) > 0$ for $p \in ( \sqrt{2(\omega - \gamma^2)}, \sqrt{2\omega} ]$ by \eqref{eqVarQ2}. This concludes the proof.
\end{proof}

We define the sets $\Gamma_{\omega, \gamma}^1$, $\Gamma_{\omega, \gamma}^2$ and $\Gamma_{\omega, \gamma}^3$ by
\begin{equation} \label{eqGammaCurve1}
    \begin{aligned}
        \Gamma_{\omega, \gamma}^1 & = \left\{ \left( p, q_{\omega, \gamma, +}(p) \right) : p \in \left( 0, \sqrt{2\omega} \right]  \right\}, \\
        \Gamma_{\omega, \gamma}^2 & = \left\{ \left( p, q_{\omega, \gamma, -}(p) \right) : p \in \left[ \sqrt{2(\omega - \gamma^2)}, \sqrt{2\omega} \right] \right\}, \\
        \Gamma_{\omega, \gamma}^3 & = \left\{ \left( p, q_{\omega, \gamma, -}(p) \right) : p \in \left( 0, \sqrt{2(\omega - \gamma^2)} \right] \right\}.
    \end{aligned}
\end{equation}
By \eqref{eqSignQ}, we immediately have
\begin{equation} \label{eqGammaCurve2}
    \Gamma_{\omega, \gamma} = \Gamma_{\omega, \gamma}^1 \cup \Gamma_{\omega, \gamma}^2 \cup \Gamma_{\omega, \gamma}^3.
\end{equation}

Figure \ref{figGammaCurve} illustrates a $\Gamma_{\omega}$-curve of level $\gamma$, and the $\mathcal{E}_{\omega}$-curve of level $0$. The sets $\Gamma_{\omega, \gamma}^1$,     $\Gamma_{\omega, \gamma}^2$ and $\Gamma_{\omega, \gamma}^3$ defined in \eqref{eqGammaCurve1} have been represented. The set of parameters chosen is $\omega =1, \gamma = 1/2$.

\begin{figure}[ht]
    \centering
    \includegraphics[width=0.5\linewidth]{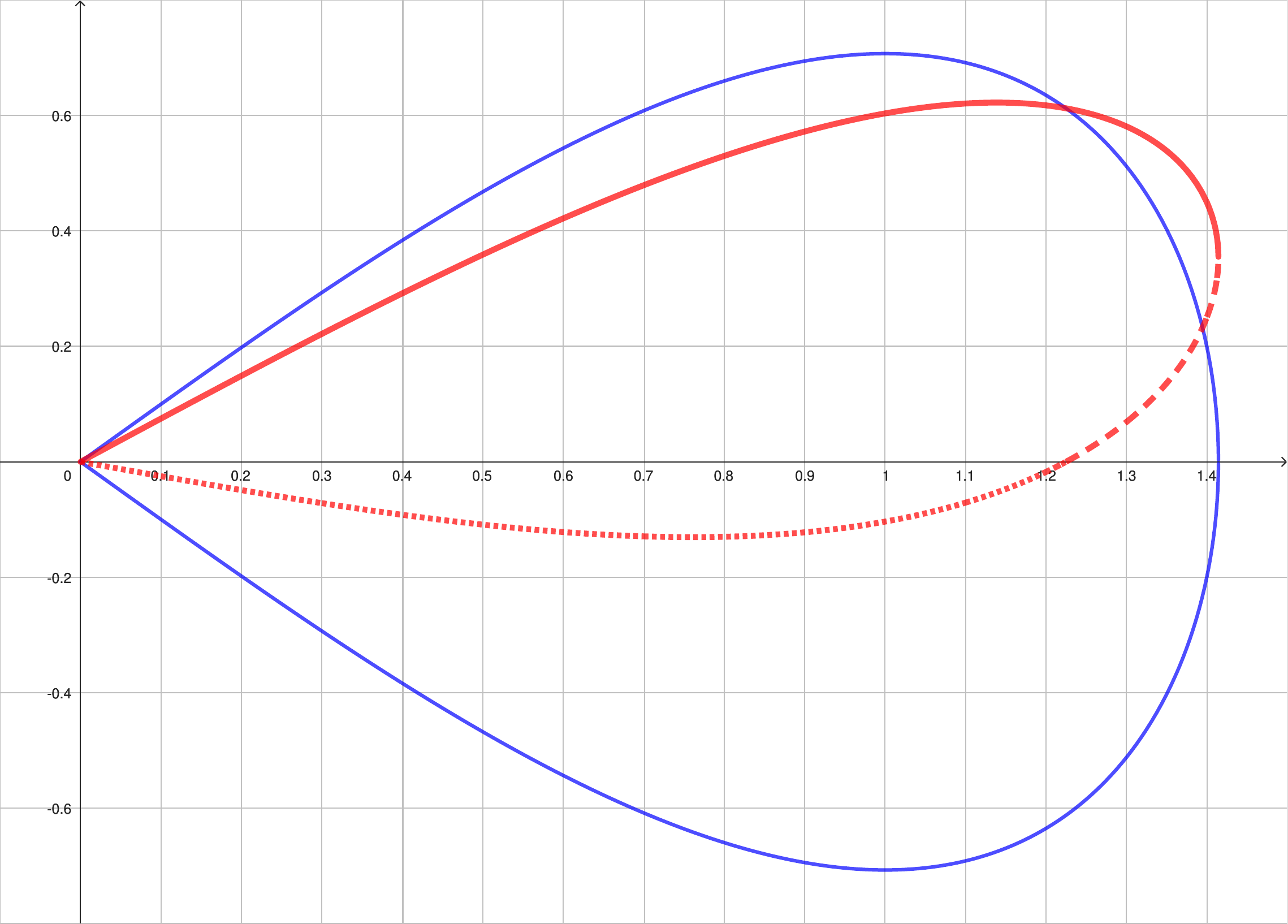}
    \caption{$\mathcal E_\omega$-curve of level $0$ (blue) and $\Gamma_\omega$-level curve of level $\gamma$ (red); $\Gamma_{\omega, \gamma}^1$ is the continuous curve, $\Gamma_{\omega, \gamma}^2$ is the large-dashed curve and $\Gamma_{\omega, \gamma}^3$ is the small-dashed curve. The parameters are $\omega = 1$ and $\gamma = 1/2$.}
    \label{figGammaCurve}
\end{figure}

Let $u = (u_c, u_h)$ be a solution of \eqref{eqNLS} such that $u_c$ is even on the compact edge $c$. The following lemma states that the point $(u_c(-L), u_c'(-L))$ belongs to the $\Gamma_\omega$-curve of level $\gamma$. Furthermore, the shape of $u_h$ is determined by the position of the point $(u_c(-L), u_c'(-L))$ on the phase portrait. The terms tail-type solution, half soliton-type solution and bump-type solution have been defined in \eqref{eqSolSech}.

\begin{lemma} \label{lemShapeLine}
    Let $\omega > 0$, $0 < \gamma < \sqrt{\omega}$ and $L > 0$ be fixed. Assume that $u = (u_c, u_h)$ is a solution of \eqref{eqNLS} such that $u_c$ is even on $c$. Then,
    \begin{equation} \label{eqShapeEvenCurve}
        (u_c(-L), u_c'(-L)) \in \Gamma_{\omega, \gamma}.
    \end{equation}
    Furthermore, the following property holds:
    \begin{enumerate}[label=(\roman*)]
        \item if $(u_c(-L), u_c'(-L)) \in \Gamma_{\omega, \gamma}^1 \setminus \{ (\sqrt{2\omega}, \gamma \omega / \sqrt{2}) \} $, then $u_h$ is a tail-type solution;
        \item if $(u_c(-L), u_c'(-L)) = (\sqrt{2\omega}, \gamma \omega / \sqrt{2})$, then $u_h$ is a half soliton-type solution;
        \item if $(u_c(-L), u_c'(-L)) \in \Gamma_{\omega, \gamma}^2 \cup \Gamma_{\omega, \gamma}^3 \setminus \{ (\sqrt{2\omega}, \gamma \omega / \sqrt{2}) \}$, then $u_h$ is a bump-type solution.
    \end{enumerate}
\end{lemma}

\begin{proof}
    By Remark \ref{rkSech}, $u_h$ is a hyperbolic secant of type \eqref{eqSolSech} with a translation parameter $b \in \mathbb R$. In particular, we have
    \begin{equation} \label{eqDeltaCond1}
        u_h'(0) =
        \begin{cases}
            -\sqrt{\omega} u_c(-L) \sqrt{1 - \frac{u_c(-L)^2}{2\omega} } & \text{ if and only if } b > 0, \\
            0 & \text{ if and only if } b = 0, \\
            \sqrt{\omega} u_c(-L) \sqrt{1 - \frac{u_c(-L)^2}{2\omega} } & \text{ if and only if } b < 0.
        \end{cases}
    \end{equation}
    by continuity at the vertex. Furthermore, the $\delta$-vertex condition and the fact that $u_c$ is even imposes that
    \begin{equation} \label{eqDeltaCond2}
        2 u_c'(-L) + u_h'(0) = \gamma u_c(-L).
    \end{equation}
    We deduce that
    \begin{equation*}
        u_c'(-L) =
        \begin{cases}
            q_{\omega, \gamma, +} (u_c(-L)) & \text{ if and only if } b > 0, \\
            \gamma \omega / \sqrt{2} & \text{ if and only if } b = 0, \\
            q_{\omega, \gamma, -} (u_c(-L)) & \text{ if and only if } b < 0.
        \end{cases}
    \end{equation*}
    by \eqref{eqDeltaCond1} and \eqref{eqDeltaCond2} and this concludes the proof.
\end{proof}

We now study the existence of solutions of \eqref{eqNLS} such that $u_c$ is even on the compact edge $c$. The following lemma gives the relation between the length $L$ of the compact edge $c$ and the period functions $\mathcal T_+$ and $\mathcal T_-$ defined in \eqref{eqPeriodFunctions}.

\begin{lemma} \label{lemLengthCircleEven}
    Let $\omega > 0$, $0 < \gamma < \sqrt{\omega}$ and $L > 0$ be fixed. Assume that $u = (u_c, u_h)$ is a solution of \eqref{eqNLS} such that $u_c$ is even on $c$. Then, the following property holds:
    \begin{enumerate}[label=(\roman*)]
        \item if $(u_c(-L), u_c'(-L)) \in (\Gamma_{\omega, \gamma}^1 \cup \Gamma_{\omega, \gamma}^2) \setminus {(\sqrt{2(\omega - \gamma^2)}, 0)}$, then there exists $n \in \mathbb N$ such that
        \[
            L = (n+1) \mathcal T_+ \left( u_c(-L), u_c'(-L) \right) + n \mathcal T_- \left( u_c(-L), u_c'(-L) \right);
        \]
        \item if $(u_c(-L), u_c'(-L)) \in \Gamma_{\omega, \gamma}^3 \setminus {(\sqrt{2(\omega - \gamma^2)}, 0)}$, then there exists $n \in \mathbb N$ such that
        \[
            L = n \mathcal T_+ \left( u_c(-L), u_c'(-L) \right) + (n+1) \mathcal T_- \left( u_c(-L), u_c'(-L) \right);
        \]
        \item if $(u_c(-L), u_c'(-L)) = (\sqrt{2(\omega - \gamma^2)}, 0)$, , then there exists $n \in \mathbb N$ such that
        \[
            L = n \mathcal T_+ \left(\sqrt{2(\omega - \gamma^2)}, 0 \right) + n\mathcal T_- \left( \sqrt{2(\omega - \gamma^2)}, 0\right).
        \]
    \end{enumerate}
\end{lemma}

\begin{proof}
    Let $T$ be the period of $u_c$ and $n = \lfloor 2L / T \rfloor \in \mathbb N$ be the number of periods that $u_c$ crosses on $[-L,L]$ and $E = \mathcal E_\omega(u_c(-L), u_c'(-L))$. Then the solution $u_c$ moves along the $\mathcal E_\omega$-curve of level $E$ in the direction of the flow represented in Figure \ref{figECurve}, starting from the point $(u_c(-L), u_c'(-L))$, and, as $u_c$ is even, stops at the point $(u_c(L), u_c'(L)) = (u_c(-L), -u_c'(-L))$.
    
    If $(u_c(-L), u_c'(-L)) \in \Gamma_{\omega, \gamma}^1 \cup \Gamma_{\omega, \gamma}^2 \setminus {(\sqrt{2(\omega - \gamma^2)}, 0)}$, then in particular $u_c'(-L) \geq 0$, so $u_c'(L) = -u_c'(-L) \leq 0$ and the point $(u_c(x),u_c'(x))$ travels along $\mathcal E_{\omega, E}^{u_c(-L),2}$ one more time than along $\mathcal E_{\omega, E}^{u_c(-L), 1} \cup \mathcal E_{\omega, E}^{u_c(-L), 3}$ as $x$ goes from $-L$ to $L$. Thus, there exists $n \in \mathbb N$ such that
    \[
        L = (n+1) \mathcal T_+ \left( u_c(-L), u_c'(-L) \right) + n \mathcal T_- \left( u_c(-L), u_c'(-L) \right)
    \]
    by definition of $\mathcal T_+$ and $\mathcal T_-$.
    
    If $(u_c(-L), u_c'(-L)) \in \Gamma_{\omega, \gamma}^3 \setminus {(\sqrt{2(\omega - \gamma^2)}, 0)}$, then in particular $u_c'(-L) \leq 0$ so that $u_c'(L) = -u_c'(-L) \geq 0$ and the point $(u_c(x),u_c'(x))$ travels along $\mathcal E_{\omega, E}^{u_c(-L), 1} \cup \mathcal E_{\omega, E}^{u_c(-L), 3}$ one more time than along $\mathcal E_{\omega, E}^{u_c(-L),2}$ as $x$ goes from $-L$ to $L$. Thus, there exists $n \in \mathbb N$ such that
    \[
        L = n \mathcal T_+ \left( u_c(-L), u_c'(-L) \right) + (n+1) \mathcal T_- \left( u_c(-L), u_c'(-L) \right)
    \]
    by definition of $\mathcal T_+$ and $\mathcal T_-$. Point (iii) will be proved in Proposition \ref{propShapeEven}. This concludes the proof.
\end{proof}

We then show that there exists a finite number of positive solutions of \eqref{eqNLS} which are even on the compact edge $c$. In the proof, we first build such a solution by studying the $\mathcal E_\omega$ and $\Gamma_\omega$ curve and we deduce that there is a finite number of them by studying the period functions $\mathcal T_+$ and $\mathcal T_-$.

\begin{lemma} \label{lemShapeFinite}
    Let $\omega > 0$, $0 < \gamma < \sqrt{\omega}$ and $L > 0$ be fixed. Then there exists a non-negative solution $u = (u_c, u_h)$ of \eqref{eqNLS} such that $u_c$ is even on $c$. If $\gamma \neq \sqrt{\omega / 2}$, there exists a finite number of those solutions.
\end{lemma}

\begin{proof}
    We divide the proof into two steps. In the first one, we build a solution of \eqref{eqNLS} which is even on $c$. In the second one, we prove that there exists a finite number of them for $\gamma \neq \sqrt{\omega / 2}$.
    
    \textit{Step 1:} We study the period function $\mathcal T_+$ on the set $\Gamma_{\omega, \gamma}^1 \cup \Gamma_{\omega, \gamma}^2$. As $q_{\omega, \gamma, -}(p) \to 0$ as $p \to 0^+$, we have
    \[
        \mathcal T_+ \left( p, q_{\omega, \gamma, -}(p) \right) \to \infty
    \]
    as $p \to 0^+$ by Lemma \ref{lemPeriodFunc}. Furthermore,
    \[
        \mathcal T_+ \left( \sqrt{2(\omega - \gamma^2)}, q_{\omega, \gamma, -} \left( \sqrt{2 \left( \omega - \gamma^2 \right)} \right) \right) = 0.
    \]
    As $\mathcal T_+$ is continuous, there exists $(\tilde{p}, \tilde{q}) \in \Gamma_{\omega, \gamma}^1 \cup \Gamma_{\omega, \gamma}^2$ such that
    \[
        \mathcal T_+ \left( \tilde{p}, \tilde{q} \right) = L.
    \]
    Thus, considering the sets
    \[
        \mathcal E_{\omega, \mathcal E_\omega (\tilde{p}, \tilde{q})}^{\tilde{p}, 2} \cup \mathcal E_{\omega, 0}^{\tilde{p}, 1} \quad \text{if } (\tilde{p}, \tilde{q}) \in \Gamma_{\omega, \gamma}^1
    \]
    or
    \[
        \mathcal E_{\omega, \mathcal E_\omega (\tilde{p}, \tilde{q})}^{\tilde{p}, 2} \cup \mathcal E_{\omega, 0}^{\tilde{p}, 1} \cup \mathcal E_{\omega, 0}^{\tilde{p}, 2} \quad \text{if } (\tilde{p}, \tilde{q}) \in \Gamma_{\omega, \gamma}^2
    \]
    and the flow of the phase portrait, we build $u = (u_c, u_h) \in H^2(\mathcal{G}_L)$ such that $u_c$ satisfies
    \begin{align*}
        & u_c(-L) = u_c(L) = \tilde{p}, \\
        & u_c'(-L) = -u_c'(L) = \tilde{q}, \\
        & \mathcal E_{\omega} (u_c(x), u_c'(x)) = \mathcal E_{\omega} \left( \tilde{p}, \tilde{q} \right)
    \end{align*}
    for $x \in [-L, L]$ and $u_h$ satisfies
    \[
        u_h(x) = \sqrt{2 \omega} \operatorname{sech} \left( \sqrt{\omega} x + b \right),
    \]
    for $x \in [0, \infty)$, with $b \in \mathbb R$ which satisfies
    \[
        \sqrt{2 \omega} \operatorname{sech} \left( b \right) = \tilde{p}.
    \]
    Observe that $u_c$ is even on $c$. Furthermore, as $(\tilde{p}, \tilde{q}) \in \Gamma_{\omega, \gamma}$, such a $u$ satisfies the $\delta$-vertex conditions by Lemma \ref{lemShapeLine}. Thus, $u$ is a solution of \eqref{eqNLS}. This concludes the first step of the proof.
    
    \textit{Step 2:} We now show that there exists a finite number of solutions of \eqref{eqNLS} which are even on $c$. Assume by contradiction that there exists an infinite number of solutions of \eqref{eqNLS} which are even on $c$. Then, by Lemma \ref{lemLengthCircleEven}, there exist sequences $(p_n, q_n)_{n \in \mathbb N^*} \subset \Gamma_{\omega, \gamma}$ and $(N(n))_{n \in \mathbb N} \subset \mathbb N$ such that $N(n) \to \infty$ as $n \to \infty$ and, for $n \in \mathbb N$, either
    \[
        L = (N(n) + 1) \mathcal T_+ \left( p_n, q_n \right) + N(n) \mathcal T_- \left( p_n, q_n \right)
    \]
    or
    \[
        L = N(n) \mathcal T_+ \left( p_n, q_n \right) + (N(n) + 1) \mathcal T_- \left( p_n, q_n \right).
    \]
    Thus,
    \[
        \mathcal T_+ \left( p_n, q_n \right) \to 0 \text{ and } \mathcal T_- \left( p_n, q_n \right) \to 0
    \]
    as $n \to \infty$. By Lemma \ref{lemPeriodFunc} and continuity of $\mathcal T_+$ and $\mathcal T_-$, we thus have
    \[
        (p_n, q_n) \to \left( \sqrt{\omega}, 0 \right)
    \]
    as $n \to \infty$. However, as $\gamma \neq \sqrt{\omega / 2}$, the point $(\sqrt{\omega}, 0)$ does not belong to $\Gamma_{\omega, \gamma}$, so that
    \[
        \inf \left\{ \left| \left( \sqrt{\omega}, 0 \right) - (p, q) \right|: (p, q) \in \Gamma_{\omega, \gamma} \right\} > 0.
    \]
    As $(p_n, q_n)_{n \in \mathbb N} \subset \Gamma_{\omega, \gamma}$, this is a contradiction and concludes the second step of the proof.
\end{proof}

\begin{figure}[ht]
    \centering
    \begin{subfigure}{0.4\textwidth}
        \centering
        \includegraphics[width=\linewidth]{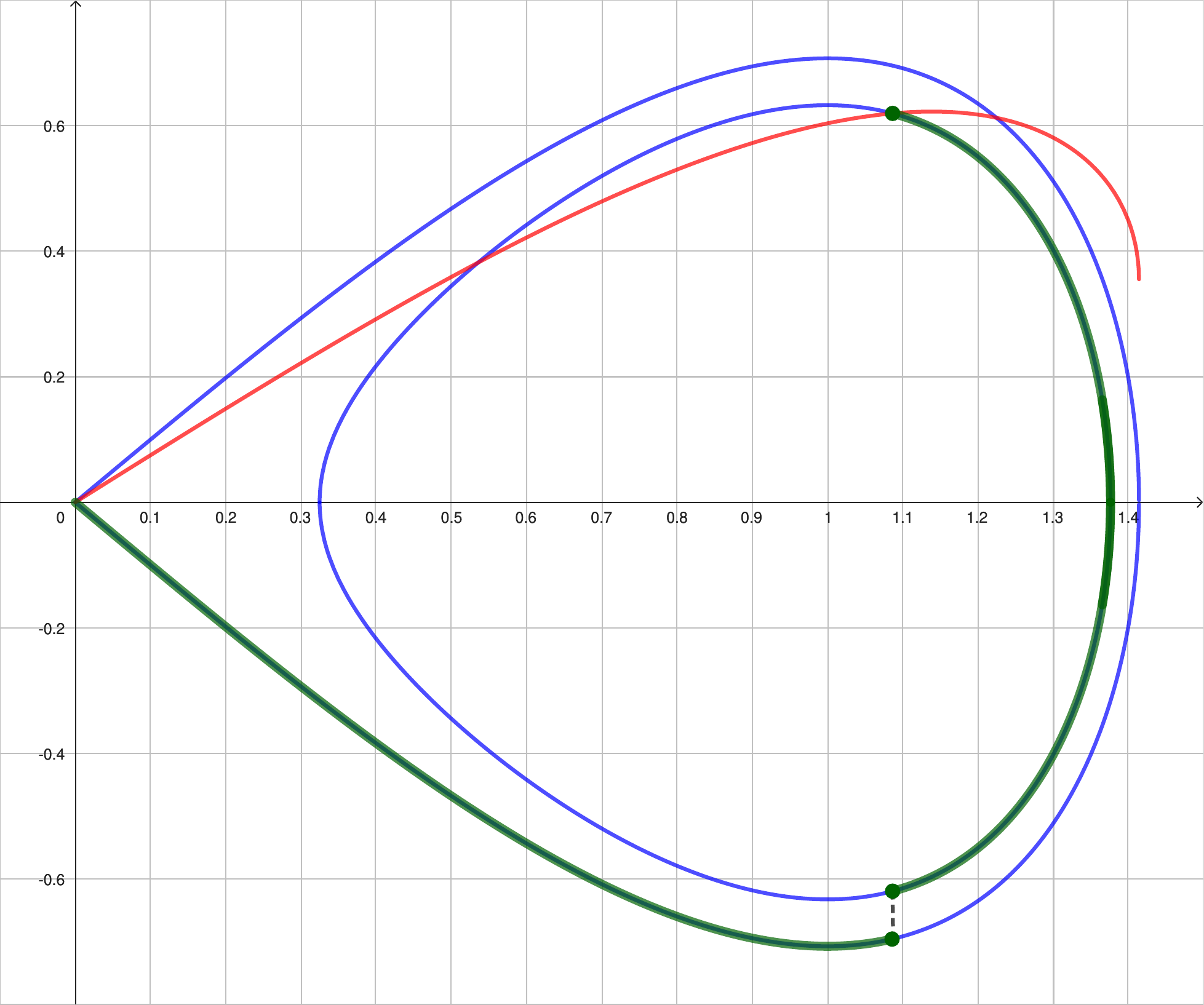}
        \caption{Case $(\tilde{p}, \tilde{q}) \in \Gamma_{\omega, \gamma}^1$.}
        \label{figSolution1}
    \end{subfigure}
    \hfill
    \begin{subfigure}{0.4\textwidth}
        \centering
        \includegraphics[width=\linewidth]{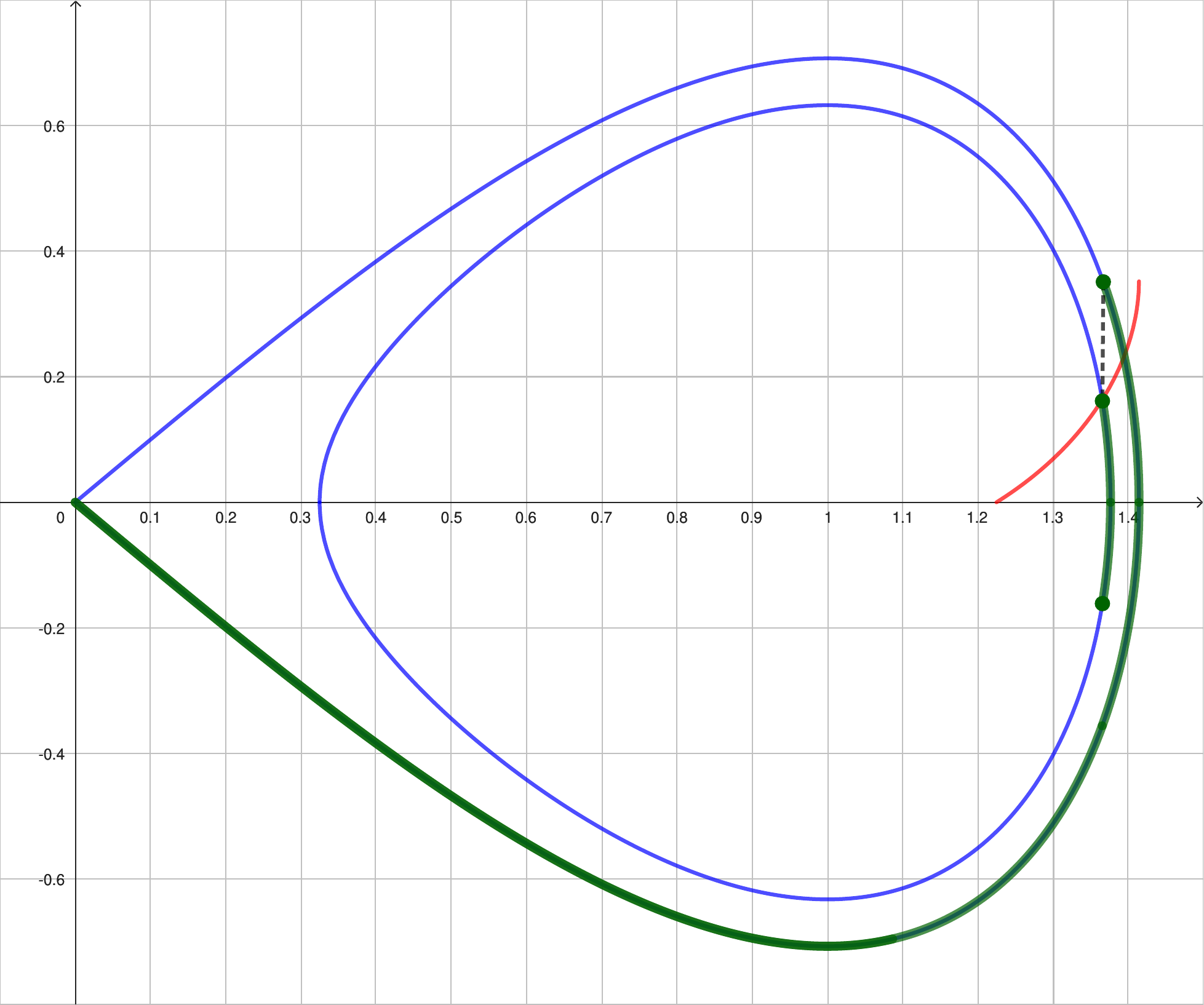}
        \caption{Case $(\tilde{p}, \tilde{q}) \in \Gamma_{\omega, \gamma}^2$.}
        \label{figSolution2}
    \end{subfigure}
    \caption{$\mathcal E_\omega$-curve of level $0$ and $\mathcal E_\omega (\tilde{p}, \tilde{q})$ (blue); $\Gamma_{\omega, \gamma}^1$ (red, left) and $\Gamma_{\omega, \gamma}^2$ (red, right); and the solution built in Step $1$ of the proof of Lemma \ref{lemShapeFinite} (green).}
\end{figure}

Figure \ref{figSolution1} and \ref{figSolution2} represent the construction of the solution in Step $1$ of the proof of Lemma \ref{lemShapeFinite}.

We can henceforth classify the shape of the positive solutions $u = (u_c, u_h)$ of \eqref{eqNLS} such that $u_c$ is even under the conditions of the Proposition \ref{propExistence}. The terms dnoidal-type solution and cnoidal-type solution have been defined in \eqref{eqSolDn} and \eqref{eqSolDn}.

\begin{proposition} \label{propShapeEven}
    Let $\omega > 0$, $0 < \gamma < \sqrt{\omega}$ and $L > 0$ be fixed. Let $u = (u_c, u_h)$ be a non-negative solution of \eqref{eqNLS} such that $u_c$ is even on $c$. The following properties hold:
    \begin{enumerate}[label=(\roman*)]
        \item if $(u_c(-L), u_c'(-L)) \in \Gamma_{\omega, \gamma}^1$, then $u_h$ is a tail-type solution and:
        \begin{enumerate}
            \item if $u_c(-L) < \sqrt{2(\omega - \gamma^2)}$, then $u_c$ is a dnoidal-type solution,
            \item if $u_c(-L) = \sqrt{2(\omega - \gamma^2)}$, then $u_c$ is a hyperbolic secant,
            \item if $\sqrt{2(\omega - \gamma^2)} < u_c(-L) < \sqrt{2\omega}$, then $u_c$ is a cnoidal-type solution;
        \end{enumerate}
        \item if $(u_c(-L), u_c'(-L)) = \left( \sqrt{2\omega}, q_{\omega, \gamma, +} \left( \sqrt{2\omega} \right) \right)$, then $u_h$ is a half soliton-type solution and $u_c$ is a cnoidal-type solution;
        \item if $(u_c(-L), u_c'(-L)) \in \Gamma_{\omega, \gamma}^2$, then $u_h$ is a bump-type solution and:
        \begin{enumerate}
            \item if $u_c(-L) > \sqrt{2(\omega - \gamma^2/9)}$, then $u_c$ is a cnoidal-type solution,
            \item if $u_c(-L) = \sqrt{2(\omega - \gamma^2/9)}$, then $u_c$ is a hyperbolic secant,
            \item if $\sqrt{2(\omega - \gamma^2)} < u_c(-L) < \sqrt{2(\omega - \gamma^2/9)}$, then $u_c$ is a dnoidal-type solution;
        \end{enumerate}
        \item if $(u_c(-L), u_c'(-L)) = ( \sqrt{2(\omega - \gamma^2)}, 0 )$, then $u_h$ is a bump-type solution and $u_c$ is a dnoidal-type solution; in particular, if $\gamma = \sqrt{\omega / 2}$, then $u_c \equiv \sqrt{\omega}$;
        \item if $(u_c(-L), u_c'(-L)) \in \Gamma_{\omega, \gamma}^3$, then $u_h$ is a bump-type solution and $u_c$ is a dnoidal-type solution.
    \end{enumerate}
    Furthermore, in the cases $(i)$-$(a)$, $(iii)$-$(c)$ and $(v)$, the translation parameter $a$ of $u_c$ is either $a = 0$ or $a = \sqrt{(1 - 2k^2)/\omega}K(k)$;
    in the cases $(i)$-$(b)$ and $(iii)$-$(b)$, the translation parameter of $u_c$ is $a = 0$;
    in the cases $(i)$-$(c)$, $(ii)$ and $(iii)$-$(a)$, the translation parameter of $u_c$ is $a = 0$ and its parameter $k$ satisfies $\sqrt{(1-2k^2)/\omega} K(k) > L$;
    and in the case $(iv)$, the translation parameter $b$ of $u_h$ is given by $b = -\operatorname{arctanh} ( \gamma/\sqrt{\omega} )$ and the translation parameter $a$ of $u_c$ is either $a = 0$ or $a = \sqrt{(1 - 2k^2)/\omega}K(k)$ and satisfies $L = n\sqrt{(2 - k^2) / \omega}K(k)$ for some $n \in \mathbb N$.
\end{proposition}

\begin{proof}
    By Lemma \ref{lemShapeFinite}, such a solution $u$ does exist. We will show that its shape is determined by the position of the point $(u_c(-L), u_c'(-L))$ on the phase portrait. The proof is based on the following study of the position of the $\Gamma_{\omega}$-curve of level $\gamma$ relatively to the $\mathcal{E}_{\omega}$-level curve of level $0$.
    
    First, observe that if $\mathcal E_\omega(p,q) = 0$, then
    \[
        q^2 = p^2 \left( \omega - \frac{p^2}{2} \right).
    \]
    Thus, we consider the functions $F_+$, $F_-$ and $G$ given by
    \[
        F_+(p) = \sqrt{p^2 \left( \omega - \frac{p^2}{2} \right)} - q_{\omega, \gamma, +} (p), \quad F_-(p) = \sqrt{p^2 \left( \omega - \frac{p^2}{2} \right)} - q_{\omega, \gamma, -} (p)
    \]
    and
    \[
        G(p) = -\sqrt{p^2 \left( \omega - \frac{p^2}{2} \right)} - q_{\omega, \gamma, -} (p)
    \]
    for $p \in (0, \sqrt{2\omega}]$. As $\gamma < \sqrt{\omega}$, the study of those functions shows that
    \begin{align}
        & F_+(p) > 0 \text{ for } 0 < p < \sqrt{2(\omega - \gamma^2)}, & F_-(p) > 0 & \text{ for } 0 < p < \sqrt{2(\omega - \gamma^2/9)}, \label{eqShapeCircle1} \\
        & F_+(p) = 0 \text{ for } p = \sqrt{2(\omega - \gamma^2)}, & F_-(p) = 0 & \text{ for } p = \sqrt{2(\omega - \gamma^2/9)}, \label{eqShapeCircle2} \\
        & F_+(p) < 0 \text{ for } \sqrt{2(\omega - \gamma^2)} < p \leq \sqrt{2\omega}, & F_-(p) < 0 & \text{ for } \sqrt{2(\omega - \gamma^2/9)} < p \leq \sqrt{2\omega}, \label{eqShapeCircle3} \\
        & G(p) < 0 \text{ for } p \in \left( 0, \sqrt{2\omega} \right). \label{eqShapeCircle4}
    \end{align}
    We deduce that if $(u_c(-L), u_c'(-L)) = (p, q_{\omega, \gamma, +}(p))$ or $(p, q_{\omega, \gamma, -}(p))$ and $E = \mathcal E_{\omega}(u_c, u_c')$, then $u_c$ is a dnoidal-type solution in the cases \eqref{eqShapeCircle1} and \eqref{eqShapeCircle4} as $\mathcal E_{\omega}(u_c, u_c') < 0$; a hyperbolic secant in the case \eqref{eqShapeCircle2} as $\mathcal E_{\omega}(u_c, u_c') = 0$; and a cnoidal solution in the case \eqref{eqShapeCircle3} as $\mathcal E_{\omega}(u_c, u_c') > 0$.
    
    The point $(i)$-$(a)$ comes from \eqref{eqShapeCircle1} and Lemma \ref{lemShapeLine}-$(i)$; and the point $(iii)$-$(c)$ comes from \eqref{eqShapeCircle1} and Lemma \ref{lemShapeLine}-$(iii)$. The translation parameter $a$ of $u_c$ is is either $a = 0$ or $a = \sqrt{(1 - 2k^2)/\omega}K(k)$ in order for $u_c$ to be even.
    
    The point $(i)$-$(b)$ comes from \eqref{eqShapeCircle2} and Lemma \ref{lemShapeLine}-$(i)$; and the point $(iii)$-$(b)$ comes from \eqref{eqShapeCircle2} and Lemma \ref{lemShapeLine}-$(iii)$.

    The point $(i)$-$(c)$ comes from \eqref{eqShapeCircle3} and Lemma \ref{lemShapeLine}-$(i)$; the point $(ii)$ comes from \eqref{eqShapeCircle3} and Lemma \ref{lemShapeLine}-$(ii)$; and the point $(iii)$-$(a)$ comes from \eqref{eqShapeCircle3} and Lemma \ref{lemShapeLine}-$(iii)$. The translation parameter $a$ of $u_c$ is $a = 0$ and the parameter $k$ satisfies $\sqrt{(1-2k^2)/\omega} K(k) > L$ in order for $u_c$ to be non-negative and even.
    
    We prove point $(iv)$ in the following way. Assume that $(u_c(-L), u_c'(-L)) \in \Gamma_{\omega, \gamma}^2$ and $u_c(-L) = \sqrt{2(\omega - \gamma^2)}$. By continuity at the vertex, we have
    \[
        \left( u_c(-L), u_c'(-L) \right) = \left( \sqrt{2(\omega - \gamma^2)}, 0 \right) = \left( u_c(L), u_c'(L) \right).
    \]
    By definition of $\mathcal T_+$ and $\mathcal T_-$, we have that there exists $n \in \mathbb N$ such that
    \[
        L = n \left( \mathcal T_+ \left( u_c(-L), u_c'(-L) \right) + \mathcal T_- \left( u_c(-L), u_c'(-L) \right) \right).
    \]
    so the period $T$ of $u_c$ satisfies $L = nT/2$ by \eqref{eqPeriodSolution}. Thus, the parameter $k$ of $u_c$ satisfies $L = n\sqrt{(2 - k^2) / \omega}K(k)$. The parameter $a$ satisfies $a = 0$ or $a = \sqrt{(1 - 2k^2)/\omega}K(k)$ in order for $u_c$ to be even. We now consider $b \in \mathbb R$ such that
    \[
        u_h(x) = \sqrt{2\omega} \operatorname{sech} \left( \sqrt{\omega}x + b \right),
    \]
    for $x \in [0, \infty)$. As $u_c'(-L) = -u_c'(L)$, the condition 
    \[
        u_h'(0) = \gamma u_h(0)
    \]
    holds in order to satisfy $\delta$-condition. Thus, a direct computation shows that
    \[
        b = -\operatorname{arctanh} \left( \frac{\gamma}{\sqrt{\omega}} \right).
    \]
    Finally, we obtain that $u_c \equiv \sqrt{\omega}$ if $\gamma = \sqrt{\omega / 2}$ as $(\sqrt{\omega}, 0)$ is a fixed point of the phase portrait.

    The point $(v)$ comes from \eqref{eqShapeCircle4} and Lemma \ref{lemShapeLine}-$(iii)$. The translation parameter $a$ of $u_c$ is is either $a = 0$ or $a = \sqrt{(1 - 2k^2)/\omega}K(k)$ in order for $u_c$ to be even.
\end{proof}

\subsection{Non-even stationary states}

In this subsection, we study solutions which are not even on the compact edge. We define the set
\[
    \Gamma_{\omega, \gamma}^4 = \left\{ \left( \sqrt{2(\omega - \gamma^2)}, q \right): q \in \mathbb R \setminus \{ 0 \} \text{ and } \mathcal E_\omega \left( \sqrt{2(\omega - \gamma^2)}, q \right) < 0 \right\}.
\]

For a solution $u = (u_c, u_h)$ of \eqref{eqNLS} such that $u_c$ is non-even, the following lemma gives the localization of $(u_c(-L), u_c'(-L))$ on the phase portrait and the shape of $u_h$. 

\begin{lemma} \label{lemNonEvenShape}
    Let $\omega > 0$, $\gamma < \sqrt{\omega}$ and $L > 0$ be fixed. Assume that $u = (u_c, u_h)$ is a solution of \eqref{eqNLS} such that $u_c$ is not even on $c$. Then,
    \begin{equation} \label{eqNonEvenVertex}
        \left( u_c(-L), u_c'(-L) \right) = \left( u_c(L), u_c'(L) \right)
    \end{equation}
    Furthermore,
    \begin{equation} \label{eqNonEvenShape}
        (u_c(-L), u_c'(-L)) \in \Gamma_{\omega, \gamma}^4.
    \end{equation}
    and, for $x \in [0, \infty)$,
    \begin{equation} \label{eqNonEvenLine}
        u_h(x) = \sqrt{2\omega} \operatorname{sech}\left( \sqrt{\omega} x - \operatorname{arctanh} \left( \frac{\gamma}{\sqrt{\omega}} \right) \right).
    \end{equation}
\end{lemma}

\begin{proof}
    Let $u = (u_c, u_h)$ is a solution of \eqref{eqNLS} such that $u_c$ is not even. By continuity at the vertex and the fact that $\mathcal E_\omega(u_c(x), u_c'(x))$ is constant for $x \in [-L, L]$, we have that either $(u_c(-L), u_c'(-L)) = (u_c(L), -u_c'(L))$ or $(u_c(-L), u_c'(-L)) = (u_c(L), u_c'(L))$. Assume by contradiction that $u_c'(-L) = -u_c'(L)$ and let $x \in (0, L)$. Then, by uniqueness of the Cauchy problem and invariance of the equation under the symmetry $x \mapsto -x$, we obtain that the point $(u_c(-x), u_c'(-x))$ is the symmetric of $(u_c(x), u_c'(x))$ with respect to the $p$-axis. Hence, $u_c(x) = u_c(-x)$ for $x \in [-L, L]$ and $u_c$ is even, which is a contradiction. This proves \eqref{eqNonEvenVertex}.

    As \eqref{eqNonEvenVertex} holds, $u_h$ solves the equation \eqref{eqNLS1D} on $[0, \infty)$ with Robin-boundary condition
    \[
        u_h'(0) = \gamma u_h(0).
    \]
    Thus, a direct computation shows \eqref{eqNonEvenLine}. Thus, $u_c(-L) = \sqrt{2(\omega - \gamma^2)}$. The case where $u_c$ is an hyperbolic secant is impossible by continuity at the vertex. Furthermore, as $u_c$ is assumed non-negative and \eqref{eqNonEvenVertex} holds, the case where $u_c$ is a cnoidal-type solution is discarded. Thus, $u_c$ is a dnoidal-type solution and \eqref{eqNonEvenShape} holds. This concludes the proof.
\end{proof}

\begin{lemma} \label{lemNonEvenPeriods}
    Let $\omega > 0$, $\gamma < \sqrt{\omega}$ and $L > 0$ be fixed. Assume that $u = (u_c, u_h)$ is a solution of \eqref{eqNLS} such that $u_c$ is not even on $c$. Then, there exists $n \in \mathbb N$ such that
    \begin{equation} \label{eqNonEvenPeriods}
        L = n \left( \mathcal T_+ \left( u_c(-L), u_c'(-L) \right) + \mathcal T_- \left( u_c(-L), u_c'(-L) \right) \right).
    \end{equation}
    In particular, $L = nT/2$, where $T$ is the period of $u_c$.
\end{lemma}

\begin{proof}
     This is a direct consequence of \eqref{eqPeriodSolution} and \eqref{eqNonEvenVertex}.
\end{proof}

\begin{lemma} \label{lemNonEvenExistence}
    Let $\omega > 0$ and $\gamma < \sqrt{\omega}$ be fixed. Then, there exists $\tilde{L} = \tilde{L}(\omega, \gamma) \geq 0$ (with $\tilde{L} = 0$ if and only if $\gamma = \sqrt{\omega / 2}$) such that a solution $u = (u_c, u_h)$ of \eqref{eqNLS} with $u_c$ being non-even exists if and only if $L > \tilde{L}$. Furthermore, if $\gamma \neq \sqrt{\omega / 2}$, then there exists a finite number of those solutions.
\end{lemma}

\begin{proof}
    We divide the proof into two steps. In the first one, we prove the existence criterion and build a solution of \eqref{eqNLS} which is even on $c$ under it. In the second one, we prove that there exists a finite number of them for $\gamma \neq \sqrt{\omega / 2}$.
    
    \textit{Step 1:} Let
    \[
        \tilde{L} = \inf \left\{ \mathcal T_+(p,q) + \mathcal T_-(p,q) :(p,q) \in \Gamma_{\omega, \gamma}^4 \right\}.
    \]
    It is well-defined by Lemma \ref{lemPeriodFunc}. We have that $\tilde{L} > 0$ if $\gamma \neq \sqrt{\omega / 2}$ and $\tilde{L} = 0$ if $\gamma = \sqrt{\omega / 2}$ as the fixed point $(\sqrt{\omega}, 0)$ of the phase portrait belongs to the closure of $\Gamma_{\omega, \gamma}^4$. By Lemma \ref{lemPeriodFunc}, the function $(p,q) \mapsto \mathcal T_+(p,q) + \mathcal T_-(p,q)$ is continuous and we have
    \[
        \sup \left\{ \mathcal T_+(p,q) + \mathcal T_-(p,q) :(p,q) \in \Gamma_{\omega, \gamma}^4 \right\} = \infty.
    \]
    So, for $L > \tilde{L}$, there exists $(\tilde{p}, \tilde{q}) \in \Gamma_{\omega, \gamma}^4$ such that
    \[
        L = \mathcal T_+ \left( \tilde{p}, \tilde{q} \right) + \mathcal T_-\left( \tilde{p}, \tilde{q} \right).
    \]
    Thus, considering the set
    \[
        \mathcal E_{\omega, \mathcal E_\omega (\tilde{p}, \tilde{q})} \cup \mathcal E_{\omega, 0}^{\tilde{p}, 1} \cup \mathcal E_{\omega, 0}^{\tilde{p}, 2}
    \]
    and the flow of the phase portrait, we build $u = (u_c, u_h) \in H^2(\mathcal{G}_L)$ such that $u_c$ satisfies
    \begin{align*}
        & u_c(-L) = u_c(L) = \tilde{p}, \\
        & u_c'(-L) = -u_c'(L) = \tilde{q}, \\
        & \mathcal E_{\omega} (u_c(x), u_c'(x)) = \mathcal E_{\omega} \left( \tilde{p}, \tilde{q} \right)
    \end{align*}
    for $x \in [-L, L]$ and $u_h$ satisfies
    \[
        u_h(x) = \sqrt{2 \omega} \operatorname{sech} \left( \sqrt{\omega} x - \operatorname{arctanh} \left( \frac{\gamma}{\sqrt{\omega}} \right) \right)
    \]
    for $x \in [0, \infty)$. This concludes the first step of the proof.
    
    \textit{Step 2:} Assume $L > \tilde{L}$ and $\gamma \neq \sqrt{\omega / 2}$. Let $u = (u_c, u_h)$ be a solution of \eqref{eqNLS} such that $u_c$ is not even on $[-L,L]$. By Lemma \ref{lemNonEvenPeriods}, there exists $n \in \mathbb N$ such that
    \begin{equation*} \label{eqOddLength}
        L = n \left( \mathcal T_+ \left( \sqrt{2 \left( \omega - \gamma^2 \right)}, u_c'(-L) \right) + \mathcal T_- \left( \sqrt{2 \left( \omega - \gamma^2 \right)}, u_c'(-L) \right) \right).
    \end{equation*}
    We may reproduce the argument of Step 2 of the proof of Lemma \ref{lemShapeFinite}. Again, the assumption $\gamma \neq \sqrt{\omega/2}$ prevents the case where both period functions can converge to $0$ and the number of periods of the solution on the compact edge goes to $\infty$. This concludes the second step of the proof.
\end{proof}

\begin{figure}[ht]
    \centering
    \includegraphics[width=0.4\linewidth]{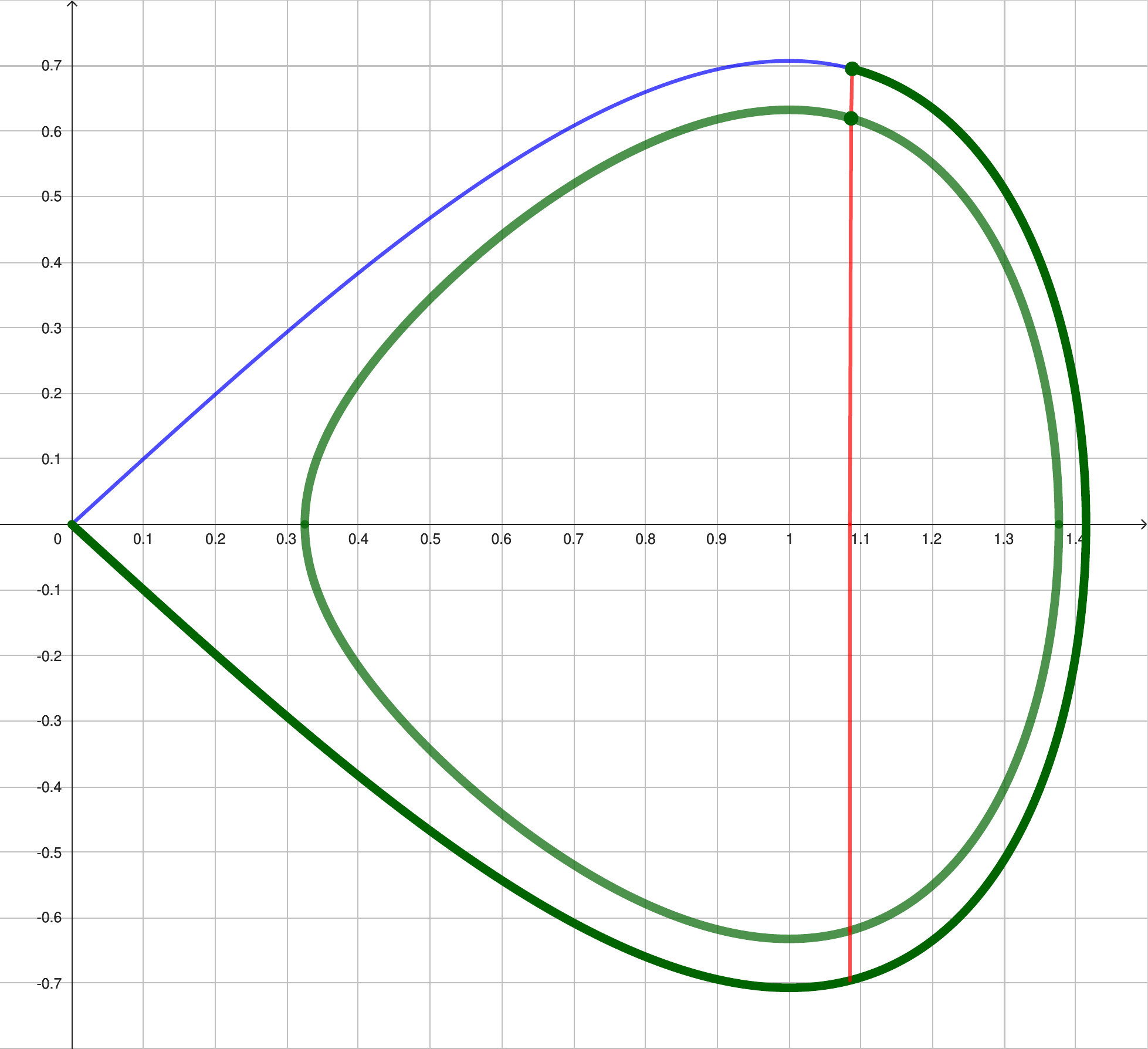}
    \caption{$\mathcal E_\omega$-curve of level $0$ (blue); $\Gamma_{\omega, \gamma}^4$ (red); and the solution built in Step $1$ of the proof of Lemma \ref{lemNonEvenExistence} (green).}
    \label{figSolution3}
\end{figure}

Figure \ref{figSolution3} represents the construction of the solution in Step $1$ of the proof of Lemma \ref{lemNonEvenExistence}.

The following proposition holds.

\begin{proposition} \label{propShapeOdd}
    Let $\omega > 0$, $\gamma < \sqrt{\omega}$ and $L > \tilde{L}$ be fixed, where $\tilde{L}$ has been defined in \ref{lemNonEvenExistence}. Let $u = (u_c, u_h)$ be a solution of \eqref{eqNLS} such that $u_c$ is not even on $[-L,L]$. Then,
    \begin{align*}
        u_c(x) & = \sqrt{\frac{2\omega}{2-k^2}} \operatorname{dn} \left( \sqrt{\frac{\omega}{2-k^2}}x + a; k \right) & \text{ for } x \in & [-L, L], \\
        u_h(x) & = \sqrt{2\omega} \operatorname{sech}\left( \sqrt{\omega} x - \operatorname{arctanh} \left( \frac{\gamma}{\sqrt{\omega}}\right) \right) & \text{ for } x \in & [0, \infty),
    \end{align*}
    with $a \in \mathbb R$ and $k \in (0,1)$ satisfying $u_c(-L) = u_c(L) = u_h(0)$ and $n\sqrt{(2 - k^2)\omega}K(k) = L$ for some $n \in \mathbb N$.
\end{proposition}

\begin{proof}
    Let $u = (u_c, u_h)$ be a non-negative solution of the equation \eqref{eqNLS} such that $u_c$ is not even on $[-L,L]$. Its existence is guaranteed by Lemma \ref{lemNonEvenExistence}. The shape of $u_h$ is given by Lemma \ref{lemNonEvenShape}. Furthermore, $u_c$ has period $T = 2\sqrt{(2 - k^2)/\omega}K(k)$ such that $L = nT / 2$ by Lemma \ref{lemNonEvenPeriods}. This concludes the proof.
\end{proof}

We conclude this section by proving Theorem \ref{thmMain}.
\begin{proof}[Proof of Theorem \ref{thmMain}]
    The proof is a combination of Proposition \ref{propExistence}, Proposition \ref{propShapeEven} and Proposition \ref{propShapeOdd}.
\end{proof}

Figure \ref{figSet1} (respectively \ref{figSet2}) shows the set $\Gamma_{\omega, \gamma}^1 \cup \Gamma_{\omega, \gamma}^2 \cup \Gamma_{\omega, \gamma}^3 \cup \Gamma_{\omega, \gamma}^4$ for $\gamma \neq \sqrt{\omega / 2}$ (respectively $\gamma = \sqrt{\omega / 2}$), where the action ground state $u$ takes its value $(u_c(-L), u_c'(-L))$.

\begin{figure}[ht]
    \centering
    \begin{subfigure}{0.4\textwidth}
        \centering
        \includegraphics[width=\linewidth]{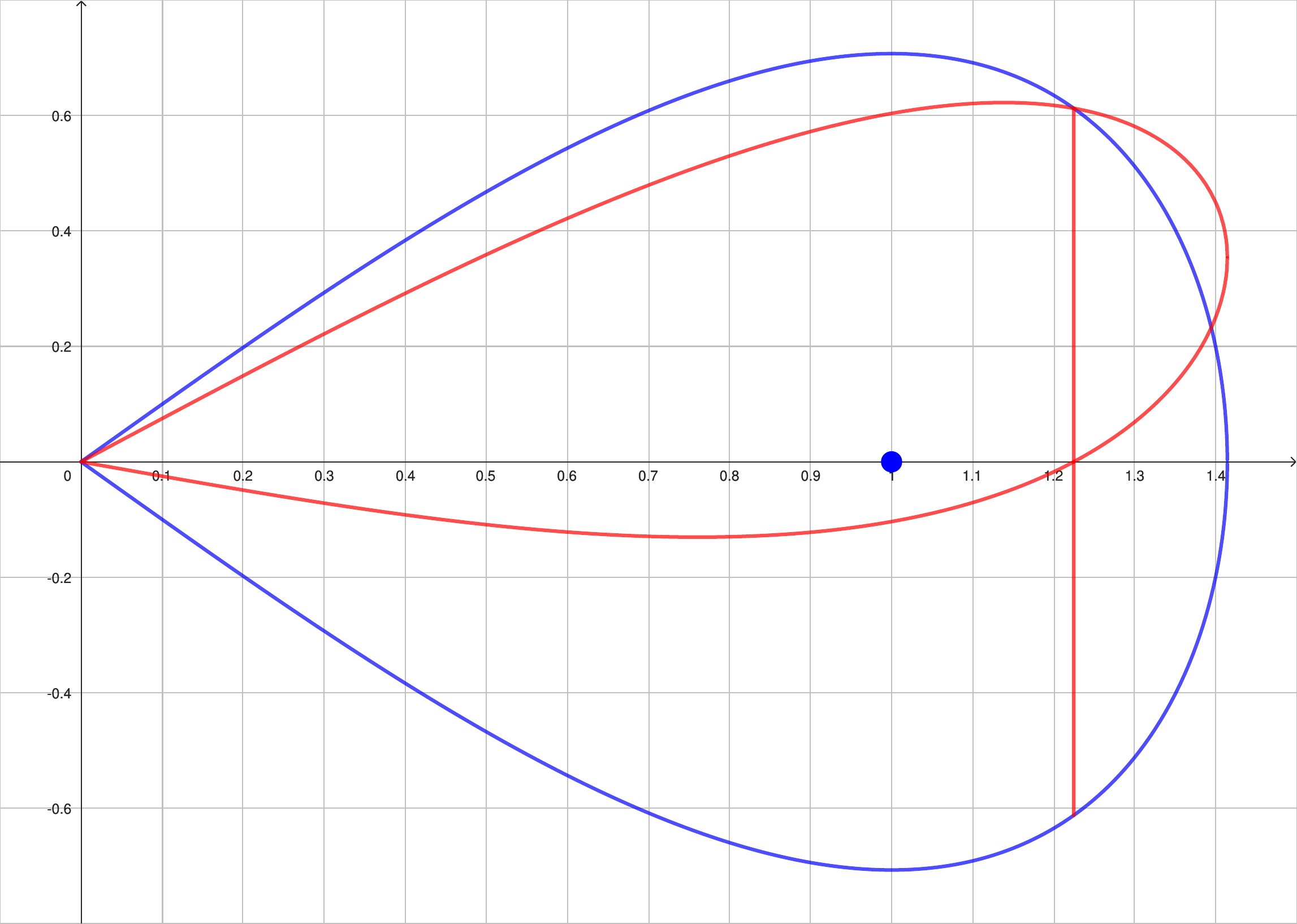}
        \caption{Case $\gamma \neq \sqrt{\omega / 2}$.}
        \label{figSet1}
    \end{subfigure}
    \hfill
    \begin{subfigure}{0.4\textwidth}
        \centering
        \includegraphics[width=\linewidth]{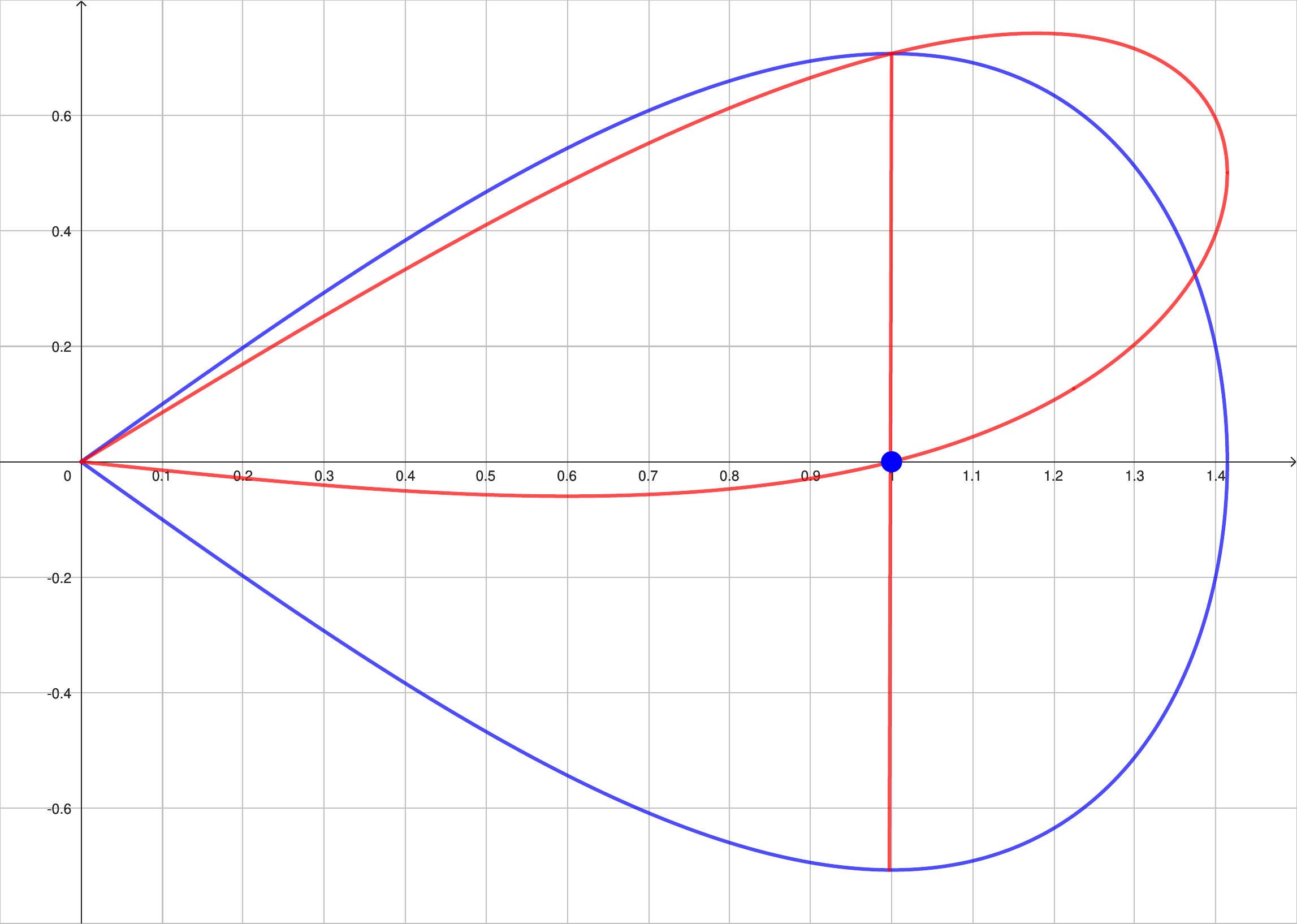}
        \caption{Case $\gamma = \sqrt{\omega / 2}$.}
        \label{figSet2}
    \end{subfigure}
    \caption{The $\mathcal E_\omega$-curve of level $0$ (blue curve); the fixed point $(\sqrt{\omega}, 0)$ (blue dot); and $\Gamma_{\omega, \gamma}^1 \cup \Gamma_{\omega, \gamma}^2 \cup \Gamma_{\omega, \gamma}^3 \cup \Gamma_{\omega, \gamma}^4$ (red).}
\end{figure}

\section{Numerical results}
\label{sec:numerics}
In this section, we discuss the numerical results we obtained through a discretized gradient flow method. We first consider the set of parameters $\omega > 0$ and $\gamma < \sqrt{\omega}$ given by Theorem \ref{thmMain} and show that a numerical action ground state $u = (u_c, u_h)$ exists for any length $L$. Furthermore, the point $(u_c(-L), u_c'(-L))$ follows the curve obtained by $\Gamma_{\omega, \gamma}^1 \cup \Gamma_{\omega, \gamma}^2$ from $(\sqrt{2(\omega - \gamma^2)}, 0)$ to $(0, 0)$ as $L$ goes from $0$ to $\infty$.

In order to minimize the action $S_{\omega, \gamma, L}$ on the Nehari manifold $\mathcal{N}_{\omega, \gamma, L}$, we use a discretized gradient flow method. It constructs a sequence $(u_n)_{n \in \mathbb N^*} \subset \mathcal N_{\omega, \gamma, L}$ which converges to a numerical action ground state $u$. This sequence is given by
\begin{equation} \label{eqGradFlow}
    \left\{
        \begin{aligned}
            & u_0 \in H_D^1(\mathcal{G}_L), \\
            &  \tilde{u}_{n+1} = u_n + \delta_t \left( \operatorname{H}_\gamma u_{n+1} + \omega u_{n+1} - |u_n|^2 u_{n+1} \right) & \text{ for } n \in \mathbb N, \\
            & u_{n+1} = \left( \frac{\| \tilde{u}_{n+1}' \|_{L^2(\mathcal{G}_L)}^2 + \omega \| \tilde{u}_{n+1} \|_{L^2(\mathcal{G}_L)}^2}{\| \tilde{u}_{n+1} \|_{L^4(\mathcal{G}_L)}^4} \right) \tilde{u}_{n+1} & \text{ for } n \in \mathbb N.
        \end{aligned}
    \right.
\end{equation}
The last line corresponds to the projection of $\tilde{u}_{n+1}$ on $\mathcal N_{\omega, \gamma, L}$. See \cite{BeDuLe22A} for more details about discretized gradient flow method applied to metric graphs. The numerical experiments were performed using the Python library Grafidi, presented in \cite{BeDuLe22B}.

We now study the shape of the numerical ground state under the conditions $\omega > 0$ and $0 < \gamma < \sqrt{\omega}$, given by Proposition \ref{propExistence}, for different lengths $L$ of the compact edge. We chose the parameters $\omega = 1$, $\gamma = 1/2$ and $L \in \{ 1/10, 1/4, 3/4, 2\}$. Observe that $\gamma < \sqrt{\omega / 2}$. The stopping criterion of the algorithm has been fixed to be
\[
    \frac{\| u_{n+1} - u_n \|_{L^2(\mathcal{G}_L)}^2}{\| u_n \|_{L^2(\mathcal{G}_L)}^2} \leq 10^{-7}.
\]
The initial datum $u_0 = (u_{0, c}, u_{0, h})$ has been fixed to be
\begin{align*}
    u_{0, c}(x) & = \cos \left( \frac{\pi}{2L} x \right) & \text{ for } x & \in [-L, L], \\
    u_{0, h}(x) & = 0 & \text{ for } x & \in [0, \infty).
\end{align*}

We finally obtained Figure \ref{figGS1}, where the method converged to the numerical action ground state $u$ in less than $250$ iterations for each case. We re-parametrized the edge $c$ from $[-L,L]$ to $[-2L,0]$ for better clarity in the figure. We observe that the numerical actions ground states are even. The translation parameter of $u_c$ is $0$, in order for $u_c$ to achieve its maximum at $x = 0$ (so $x = -L$ on the figure, as $c$ is re-parametrized). Furthermore, the parameter $k$ of $u_c$ satisfies $\sqrt{(1-2k^2)/\omega} K(k) > L$ if $u_c$ is either a dnoidal-type solution or a cnoidal-type solution.

\begin{figure}[ht]
    \centering
    \subfloat[$L = 1/10$]{\includegraphics[width=0.4\textwidth]{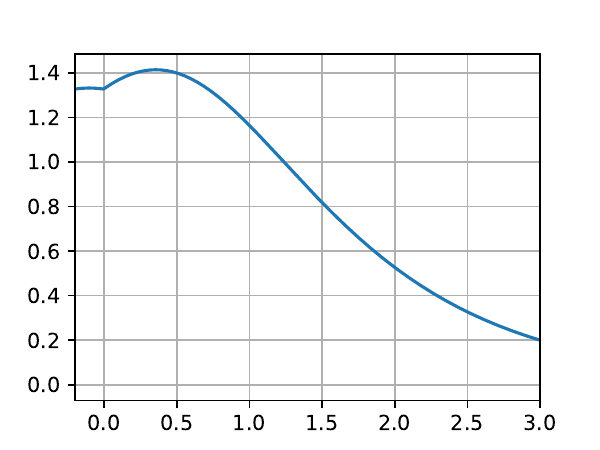}}
    \hfill
    \subfloat[$L = 1/4$]{\includegraphics[width=0.4\textwidth]{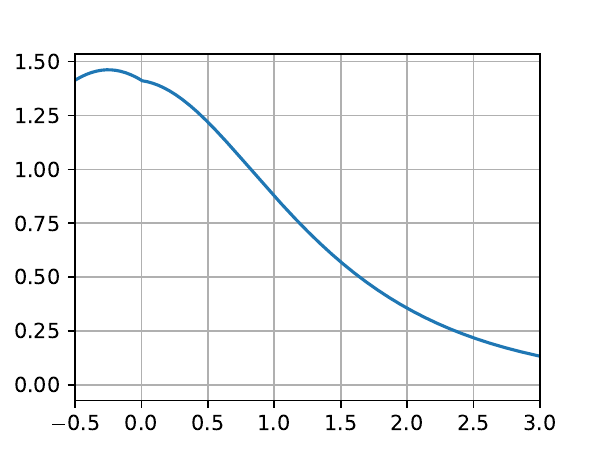}} \\
    \subfloat[$L = 3/4$]{\includegraphics[width=0.4\textwidth]{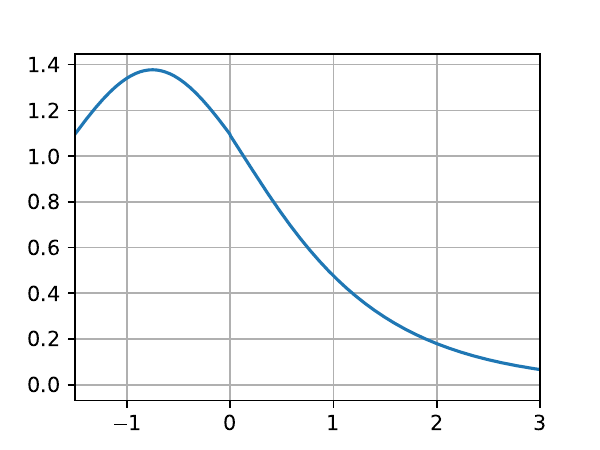}}
    \hfill
    \subfloat[$L = 2$]{\includegraphics[width=0.4\textwidth]{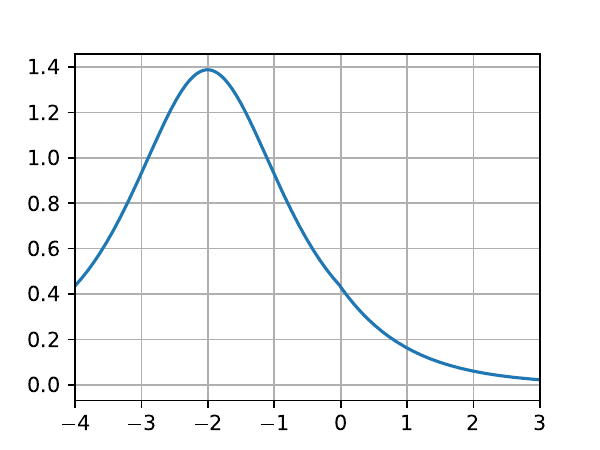}}
    \caption{Numerical action ground state for $\omega = 1$ and $\gamma = 1/2$.}
    \label{figGS1}
\end{figure}

We then plotted in Figure \ref{figGS2} the points $(u_c(-L), u_c'(-L))$ on the phase portrait for each $L$ tested. We observe that the point $(u_c(-L),u_c'(-L))$ follows the $\Gamma_\omega$-curve of level $\gamma$ from $(\sqrt{2(\omega - \gamma^2)},0)$ (which corresponds to the length $L = 0$) to $(0, 0)$ (which corresponds to the length $L = \infty$), following $\Gamma_{\omega, \gamma}^2$ and then $\Gamma_{\omega, \gamma}^1$ as $L$ grows to $\infty$. This observation has been confirmed by simulations with a larger sample of lengths (see Figure \ref{figGS2}). We therefore can conjecture that the action ground states $u = (u_c, u_h)$ has a shape among the ones given by Proposition \ref{propShapeEven} going from $(iv)$ to $(i)$ as $L$ increases from $0$ to $\infty$.

\begin{figure}[ht]
    \centering
    \includegraphics[width=0.6\linewidth]{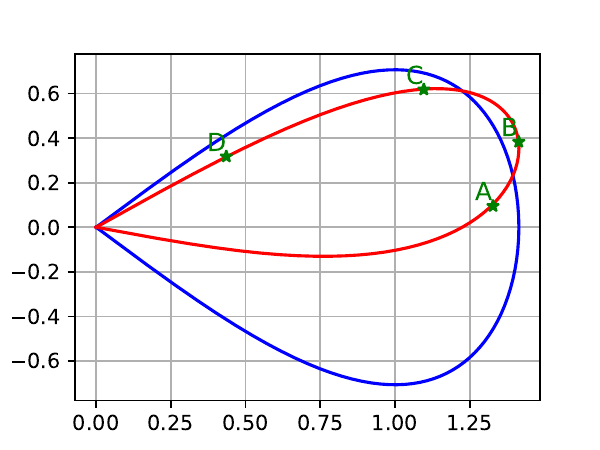}
    \caption{Plot of $(u_c(-L),u_c'(-L))$ (green) $\mathcal E_\omega$-curve of level $0$ (blue) and $\Gamma_\omega$-curve of level $\gamma$ (red); with $\omega = 1$, $\gamma = 1/2$ and $L \in \{ 1/10, 1/4, 3/4, 2\}$.}
    \label{figGS2}
\end{figure}

We performed numerical experiments with a larger sample of $0 < \gamma < \sqrt{\omega / 2}$ and $L > 0$. Those experiments showed that a numerical action ground state exists for any $\gamma$ and $L$ tested and keeps the properties cited above. See Figure \ref{figGS3} for the plot of $(u_c(-L), u_c'(-L))$ for $\omega = 1$, $\gamma = 1/4$ and $L \in [0.02, 5]$.

\begin{figure}[ht]
    \centering
    \includegraphics[width=0.6\linewidth]{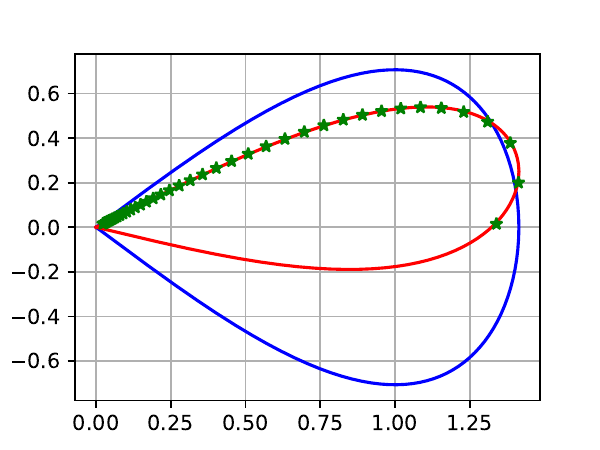}
    \caption{Plot of $(u_c(-L),u_c'(-L))$ (green) $\mathcal E_\omega$-curve of level $0$ (blue) and $\Gamma_\omega$-curve of level $\gamma$ (red); with $\omega = 1$, $\gamma = 1/4$ and $L \in [0.02, 5]$.}
    \label{figGS3}
\end{figure}

We also performed numerical experiments with parameters $\gamma$ such that $\sqrt{\omega / 2} \leq \gamma < \sqrt{\omega}$. We observed that, for $L$ large enough, a numerical action ground state exists and its shape is determined by $L$ the way described above. 

The repository \cite{Du25github} contains the codes used to perform the numerical computations and generate the Figures \ref{figGS1}, \ref{figGS2} and \ref{figGS3}.

\appendix

\section{Spectrum of the Hamiltonian operator} \label{secSpectrum}

The aim of this section is to prove Proposition \ref{propSpectrum}. In order to enhance clarity, the proof is divided in several lemmas.

As the operator $\operatorname{H}_\gamma$ is self-adjoint, we have $\sigma (\operatorname{H}_\gamma) \subset \mathbb R$. In order to compute it, we study the ordinary differential equation
\begin{equation} \label{eqSpectrum}
    -v'' - \lambda v = 0
\end{equation}
on $\mathcal{G}_L$, where $\lambda \in \mathbb R$. For $\lambda \neq 0$, solutions of \eqref{eqSpectrum} are given by $v = (v_c, v_h)$ of the form
\begin{equation} \label{eqSpectrumSol}
    \begin{aligned}
        v_c (x) & = A_c e^{i\sqrt{\lambda}x} + B_c e^{-i\sqrt{\lambda}x} & \text{ for } x & \in [-L,L], \\
        v_h (x) & = A_h e^{i\sqrt{\lambda}x} + B_h e^{-i\sqrt{\lambda}x} & \text{ for } x &\in [0, \infty),
    \end{aligned}
\end{equation}
where $A_c, B_c, A_h$ and $B_h \in \C$. In order to study the spectrum of $\operatorname{H}_\gamma$, we will consider functions $v$ of the form \eqref{eqSpectrumSol} which satisfy the continuity at the vertex
\begin{equation} \label{eqVertexCont}
    v_c(L) = v_c(-L) = v_h(0)
\end{equation}
and the jump condition on the derivatives
\begin{equation} \label{eqDiffJump}
    v_c'(-L) - v_c'(L) + v_h'(0) = \gamma v(\operatorname{v}).
\end{equation}
Let $v$ of the form \eqref{eqSpectrumSol} be such a function. It must satisfy
\begin{equation} \label{eqEssSpecCases}
    A_c \left( e^{i \sqrt{\lambda} L} - e^{-i \sqrt{\lambda} L} \right) = B_c \left( e^{i \sqrt{\lambda} L} - e^{-i \sqrt{\lambda} L} \right)
\end{equation}
by the first equality in \eqref{eqVertexCont}. The condition \eqref{eqEssSpecCases} shows that we have to distinguish two cases depending whether or not $\sqrt{\lambda}L$ is a multiple of $\pi$.

\begin{lemma} \label{lemMultPi}
    Assume that $\sqrt{\lambda}L = k \pi$ for some $k \in \mathbb N^*$. Then, $\lambda$ is an eigenvalue of $\operatorname{H}_\gamma$.
\end{lemma}

\begin{proof}
    Firstly, observe that $\lambda \geq 0$. Let $v$ be of the form \eqref{eqSpectrumSol} satisfying the condition \eqref{eqVertexCont} and \eqref{eqDiffJump}. Direct computations show that we have
    \[
        A_h + B_h = ( A_c + B_c ) e^{i k \pi}
    \]
    in order to satisfy the second equality in \eqref{eqVertexCont} and
    \[
        A_h \left( \gamma - i \sqrt{\lambda} \right) + B_h \left( \gamma + i \sqrt{\lambda} \right) = 0
    \]
    in order to satisfy \eqref{eqDiffJump}. The function $v = (v_c, v_h)$ is then given by
    \begin{equation*}
        \begin{aligned}
            v_c (x) & = A_c e^{i \sqrt{\lambda} x} + \left( -A_c + A_h e^{i k \pi} \left( 1 + \frac{i \sqrt{\lambda} - \gamma}{i \sqrt{\lambda} + \gamma} \right) \right) e^{-i \sqrt{\lambda} x} & \text{ for } x & \in [-L, L], \\
            v_h (x) & = A_h e^{i \sqrt{\lambda} x} + A_h \frac{i \sqrt{\lambda} - \gamma}{i \sqrt{\lambda} + \gamma} e^{-i \sqrt{\lambda} x} & \text{ for } x & \in [0, \infty).
        \end{aligned}
    \end{equation*}
    Setting $A_h = 0$ and $A_c = 1$, the $v$ we obtain belongs to $D(\operatorname{H}_\gamma)$, as $v(\operatorname{v}) = 0$. Hence, $\lambda \in \sigma (\operatorname{H}_\gamma)$.
\end{proof}

Unless otherwise specified, $\sqrt{\lambda} L$ is now assumed not to be a multiple of $\pi$. Let $v$ be the form \eqref{eqSpectrumSol} satisfying \eqref{eqVertexCont} and \eqref{eqDiffJump}. The first equality in \eqref{eqVertexCont} conditions then imposes
\[
    A_c = B_c,
\]
whereas the second equality in \eqref{eqVertexCont}, and \eqref{eqDiffJump} lead to that
\begin{align*}
    A_h + B_h & = 2 A_c \cos \left( \sqrt{\lambda} L \right), \\
    A_h - B_h & = 2 \gamma A_c \cos \left( \sqrt{\lambda} L \right) + 4 i A_c \sin \left( \sqrt{\lambda} L \right).
\end{align*}
The function $v$ is now given by
\begin{equation} \label{eqEigFunc}
    \begin{aligned}
        v_c(x) & = 2 A_c \cos \left( \sqrt{\lambda} x \right) & \text{ for } x & \in [-L, L], \\
        v_h(x) & = A_c \left( \cos \left( \sqrt{\lambda}L \right) \left( 1 + \frac{\gamma}{i\sqrt{\lambda}} \right) + 2 i \sin \left( \sqrt{\lambda}L \right) \right) e^{i\sqrt{\lambda} x} \\
        &\quad + A_c \left( \cos \left( \sqrt{\lambda}L \right) \left( 1 - \frac{\gamma}{i\sqrt{\lambda}} \right) - 2 i \sin \left( \sqrt{\lambda}L \right) \right) e^{-i \sqrt{\lambda}  x} & \text{ for } x & \in [0, \infty).
    \end{aligned}
\end{equation}

\begin{lemma} \label{lemPosLamb}
    Assume that $\lambda > 0$. Then, $\lambda$ belongs to $\sigma_{\operatorname{ess}}(\operatorname{H}_\gamma)$.
\end{lemma}

\begin{proof}
    Let $\lambda > 0$ such that $\lambda$ is not a multiple a $\pi$. Let $v$ be a function of the form \eqref{eqEigFunc}. We consider a function $\eta \in C^\infty(\mathbb R^+)$ which satisfies
    \begin{align*}
        \eta (x) & =
        \left\{
        \begin{aligned}
            1 & \text{ for } x \in [0, 1], \\
            0 & \text{ for } x \in [2, \infty),
        \end{aligned}
        \right.
    \end{align*}
    and $\eta$ is decreasing for $x \in (1, 2)$. We consider the sequence $(v_n = (v_{n, c}, v_{n, h}))_{n \in \mathbb N}$ defined by
    \begin{equation*}
        \begin{aligned}
            v_{n, c} (x) & = v_c (x) & \text{ for } x & \in [-L, L], \\
            v_{n, h} (x) & = \eta \left( \frac{x}{n} \right) v_h (x) & \text{ for } x & \in [0, \infty).
        \end{aligned}
    \end{equation*}
    We have
    \[
        \| v_n \|_{L^2(\mathcal{G}_L)} \to \infty \text{ as } n \to \infty,
    \]
    so we introduce the sequence $(\tilde{v}_n)_{n \in \mathbb N}$ defined by
    \[
        \tilde{v}_n = \frac{v_n}{\| v_n \|_{L^2(\mathcal{G}_L)}}
    \]
    for $n \in \mathbb N$, so that
    \begin{equation} \label{eqWeyl1}
        \left\| \tilde{v}_n \right\|_{L^2(\mathcal{G}_L)} = 1.
    \end{equation}
    Let $\phi \in L^2(\mathcal{G}_L)$, we have
    \[
        (\phi, \tilde{v}_n)_{L^2(\mathcal{G}_L)} \leq \frac{\| \phi \|_{L^2(\mathcal{G}_L)} \| \eta v \|_{L^\infty(\mathcal{G}_L)}}{\| v_n \|_{L^2(\mathcal{G}_L)}}\to 0 \text{ as } n \to \infty,
    \]
    so
    \begin{equation} \label{eqWeyl2}
        \tilde{v}_n \rightharpoonup 0 \text{ weakly in $L^2(\mathcal{G}_L)$ as $n \to \infty$}.
    \end{equation}
    Furthermore, for $n \in \mathbb N$, we have
    \[
        \tilde{v}_{n,h}'' (x) = \frac{1}{\| v_n \|_{L^2(\mathcal{G}_L)}} \left( v_h''(x) \eta \left( \frac{x}{n} \right) + \frac{2 \eta' \left( \frac{x}{n} \right) v_h'(x)}{n} + \frac{\eta'' \left( \frac{x}{n} \right) v_h(x)}{n^2} \right)
    \]
    for $x \in [-L, L]$. As $\eta' v_c'$ and $\eta'' v_c$ are bounded, $\eta' \left( \frac{\cdot}{n} \right) v_c'$ and $\eta'' \left( \frac{\cdot}{n} \right) v_c$ are supported in $[0, 2n]$, and the fact that
    \[
        \| v_{n, h} \|_{L^2(\mathbb R^+)} \to \infty \text{ as } n \to \infty,
    \]
    we got that
    \[
        \frac{2}{n} \left\| \eta' \left( \frac{\cdot}{n} \right) v_h' \right\|_{L^2(\mathbb R^+)}
        \leq \frac{2 \| \eta' v_h' \|_{L^\infty(\mathbb R^+)}}{\| v_n \|_{L^2(\mathcal{G}_L)}}
        \to 0
    \]
    and
    \[
        \frac{1}{n^2} \left\| \eta'' \left( \frac{\cdot}{n} \right) v_h \right\|_{L^2(\mathbb R^+)}
        \leq \frac{2 \| \eta'' v_h \|_{L^\infty(\mathbb R^+)}}{n \| v_n \|_{L^2(\mathcal{G}_L)}}
        \to 0
    \]
    as $n \to \infty$. Then,
    \begin{equation} \label{eqWeyl3}
        \operatorname{H}_\gamma \tilde{v}_n - \lambda \tilde{v}_n \to 0 \text{ strongly in $L^2(\mathcal{G}_L)$ as $n \to \infty$.}
    \end{equation}
    As \eqref{eqWeyl1}, \eqref{eqWeyl2} and \eqref{eqWeyl3} hold, we conclude that $\lambda \in \sigma_{\operatorname{ess}} (\operatorname{H}_\gamma)$ by Weyl's criterion.

    Finally, let $\lambda$ such that $\sqrt{\lambda}L = k \pi$ for some $k \in \mathbb N^*$. By Lemma \ref{lemMultPi}, $\lambda$ is an eigenvalue of $\operatorname{H}_\gamma$. As it is not isolated, we have that $\lambda \in \sigma_{\operatorname{ess}}(\operatorname{H}_\gamma)$.
\end{proof}

We will now consider the case where $\lambda < 0$. Observe then that $\sqrt{\lambda} = i\sqrt{|\lambda|}$. We first determine the discrete spectrum of $\operatorname{H}_\gamma$.

\begin{lemma} \label{lemEigvals}
    The operator $\operatorname{H}_\gamma$ has a negative eigenvalue if and only if $\gamma < 0$. In this case, this eigenvalue is given by $\lambda_\gamma$, where $\lambda_\gamma$ has been defined in Proposition \ref{propSpectrum}.
\end{lemma}

\begin{proof}
    Let $\lambda < 0$ and $v$ be a function of the form \eqref{eqEigFunc}. For $\lambda$ to be an eigenvalue, $v$ must belong to $H^2(\mathcal{G}_L)$. In particular, $v_h$ must belong to $L^2(\mathbb R^+)$, which implies that
    \[
        \cos \left( \sqrt{\lambda}L \right) \left( 1 - \frac{\gamma}{i\sqrt{\lambda}} \right) - 2 i \sin \left( \sqrt{\lambda}L \right) = 0.
    \]
    This equation is equivalent to \eqref{eqNegEigVal}. It has a solution $\lambda_\gamma$, which is unique and negative, if and only if $\gamma < 0$. In this case, $\lambda_\gamma$ is an eigenvalue, of geometric multiplicity $1$. By self-adjointness of $\operatorname{H}_\gamma$, its algebraic multiplicity is also $1$, so $\lambda_\gamma \in \sigma_{\operatorname{dis}}(\operatorname{H}_\gamma)$.
\end{proof}

Finally, we exclude from the spectrum the case where $\lambda < 0$, and $\lambda \neq \lambda_\gamma$ if $\gamma < 0$. Our proof is based on direct calculations for the Green's function. For a proof based on Krein's formula, see \cite[Section 5]{Ex97}.

\begin{lemma} \label{lemNegLambda}
    Assume $\lambda < 0$ and $\lambda \neq \lambda_\gamma$ if $\gamma < 0$. Then, $\lambda$ belongs to the resolvent set $\rho(\operatorname{H}_\gamma)$. Furthermore, let $f \in L^2(\mathcal{G}_L)$, $u \in D(\operatorname{H}_\gamma)$ the unique solution of
    \begin{equation} \label{eqResolventODE}
        \operatorname{H}_\gamma u - \lambda u = f,
    \end{equation}
    and $G: \mathcal{G}_L \times \mathcal{G}_L \to \C^2$ be the Green's function associated with the equation. For $\xi \in \mathcal{G}_L$ fixed, the function $G(\cdot, \xi): \mathcal{G}_L \to \C$ is written as $(G_c (\cdot, \xi), G_h(\cdot, \xi))$ and is given by
    \begin{equation} \label{eqGreenC}
        G_c (x, \xi) =
        \begin{cases}
            A(\xi) e^{\sqrt{|\lambda|}x} + B(\xi) e^{-\sqrt{|\lambda|}x}
            & \text{ if } \xi \in c, -L \leq x < \xi \\
            & \text{ or } \xi \in l, x \in c, \\
            A(\xi) e^{\sqrt{|\lambda|}x} + B(\xi) e^{-\sqrt{|\lambda|}x} + \frac{1}{\sqrt{|\lambda|}} \sinh( \sqrt{|\lambda|} (\xi - x) )
            & \text{ if } \xi \in c,  \xi < x \leq L,
        \end{cases}
    \end{equation}
    and
    \begin{equation} \label{eqGreenL}
        G_h (x, \xi) =
        \begin{cases}
            C(\xi) e^{\sqrt{|\lambda|}x} + D(\xi) e^{-\sqrt{|\lambda|}x} + \frac{1}{\sqrt{|\lambda|}} \sinh( \sqrt{|\lambda|} (\xi - x) )
            & \text{ if } \xi \in l, \xi < x < \infty \\
            & \text{ or } \xi \in c, x \in l, \\
            C(\xi) e^{\sqrt{|\lambda|}x} + D(\xi) e^{-\sqrt{|\lambda|}x}
            & \text{ if } \xi \in l,  0 \leq x < \xi,
        \end{cases}
    \end{equation}
    where
    \begin{equation} \label{eqGreenCoeff}
        \begin{aligned}
            A(\xi) = B(\xi) = \frac{e^{-\sqrt{|\lambda|} \xi}}{2 \sqrt{|\lambda|} \cosh \left( \sqrt{\lambda} L \right)} \left( 1 - \frac{\gamma}{\gamma + 2 \sqrt{|\lambda|} \left( \tanh \left( \sqrt{|\lambda|} L \right) + 1 \right)} \right), \\
            C(\xi) = \frac{e^{-\sqrt{|\lambda|} \xi}}{2\sqrt{|\lambda|}},
            \quad D(\xi) = \frac{e^{-\sqrt{|\lambda|} \xi}}{2\sqrt{|\lambda|}} \left( 1 - \frac{2 \gamma}{\gamma + 2 \sqrt{|\lambda|} \left( \tanh \left( \sqrt{|\lambda|} L \right) + 1 \right)} \right),
        \end{aligned}
    \end{equation}
    for $\xi \in \mathcal{G}_L$. The solution $u$ is given by
    \begin{equation} \label{eqGreenSol}
        \begin{aligned}
            u_c (x) & = \int_{\mathcal{G}_L} G_c (x, \xi) f (\xi) d\xi & \text{ for } x & \in [-L, L], \\
            u_h (x) & = \int_{\mathcal{G}_L} G_h (x, \xi) f (\xi) d\xi & \text{ for } x & \in [0, \infty).
        \end{aligned}
    \end{equation}
\end{lemma}

\begin{proof}
    Let $\lambda < 0$ and $\lambda \neq \lambda_\gamma$ if $\gamma < 0$. We will show that the resolvent operator $(\operatorname{H}_\gamma - \lambda \operatorname{Id})^{-1}$ is well-defined and bounded on $L^2(\mathcal{G}_L)$. To do so, we first compute the Green's function associated to the operator $(\operatorname{H}_\gamma - \lambda \operatorname{Id})$. We consider the following problem:
    \begin{align}
        & \frac{d^2}{dx^2} G_c (x, \xi) + \lambda G_c (x, \xi) =
        \begin{cases}
            \delta(x - \xi) & \text{ if } \xi \in c, \\
            0 & \text{ if } \xi \in l,
        \end{cases}
        & \text{ for } x & \in [-L,L], \label{eqGreenCircle} \\
        & \frac{d^2}{dx^2} G_h (x, \xi) + \lambda G_h (x, \xi) =
        \begin{cases}
            0 & \text{ if } \xi \in c, \\
            \delta(x - \xi) & \text{ if } \xi \in l,
        \end{cases}
        & \text{ for } x & \in (0,\infty), \label{eqGreenLine} \\
        & G_c(-L, \xi) = G_c(L, \xi) = G_h(0, \xi) & \text{ for } \xi & \in \mathcal{G}_L, \label{eqGreenContinuity} \\
        & \frac{d}{dx} G_c(-L, \xi) - \frac{d}{dx} G_c(L, \xi) + \frac{d}{dx} G_h(0, \xi) = \gamma G(\operatorname{v}, \xi) & \text{ for } \xi & \in \mathcal{G}_L. \label{eqGreenJump}
    \end{align}
    A solution $G_c(\cdot, \xi)$ of \eqref{eqGreenCircle} is of the form
    \begin{equation*}
        G_c(x, \xi) =
        \begin{cases}
            A(\xi) e^{\sqrt{|\lambda|}x} + B(\xi) e^{-\sqrt{|\lambda|}x}
            & \text{ if } \xi \in c, -L \leq x < \xi \text{ or } \xi \in l, x \in [-L, L], \\
            \tilde{A}(\xi) e^{\sqrt{|\lambda|}x} + \tilde{B}(\xi) e^{-\sqrt{|\lambda|}x}
            & \text{ if } \xi \in c,  \xi < x \leq L.
        \end{cases}
    \end{equation*}
    Let $\xi \in c$, then $G_c(\cdot, \xi)$ satisfies
    \begin{align*}
        G_c(\xi^-, \xi) - G_c(\xi^+, \xi) & = 0, \\
        \frac{d}{dx} G_c (\xi^+, \xi) - \frac{d}{dx} G_c (\xi^-, \xi) & = 1,
    \end{align*}
    which imposes that
    \begin{align*}
        \tilde{A}(\xi) & = A(\xi) - \frac{e^{-\sqrt{|\lambda|} \xi}}{2 \sqrt{|\lambda|}}, \\
        \tilde{B}(\xi) & = B(\xi) + \frac{e^{\sqrt{|\lambda|} \xi}}{2 \sqrt{|\lambda|}}.
    \end{align*}
    Thus, for $\xi \in \mathcal{G}_L$, the solution $G_c (\cdot, \xi)$ is of the form \eqref{eqGreenC}. Similar computations show that the solution $G_h (\cdot, \xi)$ is of the form \eqref{eqGreenL}.

    Let $\xi \in \mathcal{G}_L$. For $G_h(\cdot, \xi)$ to be in $L^2(\mathbb R^+)$, it must satisfy
    \[
        C(\xi) - \frac{e^{-\sqrt{|\lambda|} \xi}}{2 \sqrt{|\lambda|}} = 0.
    \]
    Furthermore, the first equality in \eqref{eqGreenContinuity} shows that
    \[
        A(\xi) = B(\xi),
    \]
    while the second equality shows that
    \[
        2 A(\xi) \cosh(\sqrt{|\lambda|L}) = D(\xi) + \frac{e^{-\sqrt{|\lambda|} \xi}}{2 \sqrt{|\lambda|}}.
    \]
    Finally, the condition \eqref{eqGreenJump} imposes that
    \[
        D(\xi) \left( \frac{\gamma}{\sqrt{|\lambda|}} + 2 \tanh(\sqrt{|\lambda|}L) + 1 \right)
        = \frac{e^{-\sqrt{|\lambda|} \xi}}{2 \sqrt{|\lambda|}} \left( -\frac{\gamma}{\sqrt{|\lambda|}} + 2 \tanh(\sqrt{|\lambda|}L) + 1 \right).
    \]
    Observe that $D(\xi)$ is well defined, as $\lambda \neq \lambda_\gamma$ for $\gamma < 0$. Thus, coefficients of $G_c (\cdot, \xi)$ and $G_h (\cdot, \xi)$ are given by \eqref{eqGreenCoeff}.
    
    Let $f$ be in $L^2(\mathcal{G}_L)$. Then, the equation
    \[
        \operatorname{H}_\gamma u - \lambda u = f
    \]
    has a $u \in D(\operatorname{H}_\gamma)$ given by \eqref{eqGreenSol}, which is unique by properties of the Green's function. Furthermore, using Hölder inequality, we have 
    \[
        \| u \|_{L^2(\mathcal{G}_L)}^2
        \leq \left( \int_{-L}^L \int_{\mathcal{G}_L} G_c(x, \xi)^2 d\xi dx + \int_0^\infty \int_{\mathcal{G}_L} G_h(x, \xi)^2 d\xi dx \right) \| f \|_{L^2(\mathcal{G}_L)}^2
        \leq C \| f \|_{L^2(\mathcal{G}_L)}^2
    \]
    for some $C > 0$. Thus, the operator $(\operatorname{H}_\gamma - \lambda \operatorname{Id})^{-1}$ is well-defined and bounded on $L^2(\mathcal{G}_L)$, so $\lambda \in \rho(\operatorname{H}_\gamma)$.
\end{proof}

We are finally able to prove Proposition \ref{propSpectrum}.

\begin{proof}[Proof of Proposition \ref{propSpectrum}]
    By self-adjointness of the operator, the spectrum $\sigma (\operatorname{H}_\gamma)$ is included in $\mathbb R$. The only remaining case to address is then $\lambda = 0$. From Lemmas \ref{lemMultPi}, \ref{lemPosLamb}, \ref{lemEigvals} and \ref{lemNegLambda}, we have that $(0, \infty) \subset \sigma_{\operatorname{ess}} (\operatorname{H}_\gamma)$ and $\sigma_{\operatorname{ess}} (\operatorname{H}_\gamma) \subset \mathbb R^+$. By closure of the essential spectrum, we have that $0 \in \sigma_{\operatorname{ess}} (\operatorname{H}_\gamma)$. This concludes the proof.
\end{proof}

\bibliographystyle{abbrv}
\bibliography{bibliography}
  
\end{document}